\documentclass[a4paper,11pt]{article}
\usepackage[latin1]{inputenc}
\usepackage[T1]{fontenc}
\usepackage{lmodern}

\usepackage{amsthm,amsmath,amsfonts,amssymb,bbm,mathrsfs}
\usepackage{mathtools}
\usepackage{enumitem}
\usepackage{url}
\usepackage{dsfont} 
\usepackage{appendix}
\usepackage{amsthm}
\usepackage{color}
\usepackage{graphicx}

\usepackage[colorlinks=true, linkcolor=blue, urlcolor=black, citecolor=blue,pdfstartview=FitH]{hyperref}

\usepackage[english]{babel}

\usepackage{caption,tikz,subfigure}
\usetikzlibrary{shapes}
\usetikzlibrary{patterns}

\usepackage[top=2.7cm, bottom=2.7cm, left=2.5cm, right=2.5cm]{geometry}

\makeatletter

\@addtoreset{equation}{section}
\makeatother

\setlist[enumerate,1]{label=(\roman*), font = \normalfont} 

\let\originalleft\left
\let\originalright\right
\renewcommand{\left}{\mathopen{}\mathclose\bgroup\originalleft}
\renewcommand{\right}{\aftergroup\egroup\originalright}

\newlength{\bibitemsep}
\newlength{\bibparskip}
\let\oldthebibliography\thebibliography
\renewcommand\thebibliography[1]{\oldthebibliography{#1}
	\setlength{\parskip}{\bibitemsep}
	\setlength{\itemsep}{\bibparskip}}

\newcommand{\N}{\mathbb{N}}

\newcommand{\R}{\mathbb{R}}

\renewcommand{\P}{\mathbb{P}}
\newcommand{\E}{\mathbb{E}}

\newcommand{\cT}{\mathcal{T}}
\newcommand{\cG}{\mathcal{G}}

\DeclareMathOperator{\arctanh}{artanh}

\newcommand{\ee}{\mathrm{e}}

\newcommand{\p}{\mathbb{P}}

\theoremstyle{plain}
\newtheorem{thm}{Theorem}
\newtheorem{prop}{Proposition}[section]
\newtheorem{lem}{Lemma}[section]
\newtheorem{cor}{Corollary}[section]

\theoremstyle{definition}

\theoremstyle{remark}
\newtheorem{rem}{Remark}[section]

\usepackage{anyfontsize}

\theoremstyle{plain}
\newtheorem{thmA}{Theorem}

\date{\today}
\title{The branching random walk in a uniform magnetic field : magnetization concentration and overlap distributions}
\author{
Olivier \textsc{Zindy}\thanks{Sorbonne Universit\'e, Sorbonne Paris Cit\'e, CNRS, Laboratoire de Probabilit\'es Statistique et Mod\'elisation, LPSM, F-75005 Paris, France. Email: \texttt{olivier.zindy@sorbonne-universite.fr} or \texttt{olivier.zindy@gmail.com}.}
}

\begin{document}

\maketitle

\begin{abstract}
Adding a uniform external magnetic field to a mean-field spin-glass model usually requires a new analysis specific to the model.
 The disordered system we consider corresponds to the Gaussian binary branching random walk (BRW) --- in the spirit of Derrida and Spohn \cite{derridaspohn88}
and studied from the statistical-physics point of view by Jagannath \cite{jagannath2016} --- and we prove that this is not the case : a single elementary
observation --- that the resulting Hamiltonian is still a BRW, now with independent, but
non-identically distributed, displacements --- allows us to use the available results for {\it general} BRW.
 Combining classical and recent results on general BRW (Biggins \cite{biggins76}, Chauvin and Rouault \cite{chauvinrouault97}, Mallein \cite{Mallein2018}), one obtains
an essentially complete picture of the model in an external magnetic field : the ground state, the free energy, the
one-step replica symmetry breaking (1-RSB) transition, and the limiting {\it genealogical} overlap distribution
with Poisson--Dirichlet statistics for the Gibbs weights, extending to this correlated setting
the classical results of Derrida \cite{Derrida1981} for the REM and of Derrida and
Gardner \cite{derridagardner86b} and Bovier and Klimovsky \cite{bovierklimovsky2008} for the GREM with a uniform external field.
We then prove a strong concentration result for the magnetization under
the Gibbs measure at low temperature, giving its explicit optimal value. It turns out that this one-replica statement is not enough to control the
classical Ising overlap between two independently sampled configurations : an
elementary counterexample (Remark \ref{rem:counter-ex}) shows that concentration of each replica's
magnetization does not, by itself, determine their joint correlation. We address this by developing a two-replica large-deviation argument
 --- resting on the classical method of {\it types} and the subadditivity of Shannon entropy, and taking the form of a uniform Chernoff
bound over the joint empirical type of a pair of configurations, matched against a
two-replica concentration estimate --- which, to our knowledge, has not been used in this setting, and
which we use to obtain the distribution of a second {\it hypercube-type} overlap, which we compare with the
results of Arguin and Kistler \cite{arguinkistler14} for the REM in a {\it random} external magnetic field.
 \end{abstract} 

\bigskip

\noindent \textbf{MSC Classes:} 60G15, 60J80, 60F10, 60G70, 82D30, 82B44, 94A17.

\medskip

\noindent \textbf{Keywords:} Spin glasses, REM, GREM, branching random walk, magnetic field, large deviations, method of types, free energy, replica symmetry breaking, magnetization concentration, overlap distribution.


\section{Introduction}

\subsection{Literature}
\label{subsec:literature}

The Branching Random Walk (BRW) was introduced to the mean-field spin-glass community by
Derrida and Spohn \cite{derridaspohn88}, who proposed it as an intermediate toy model between Derrida's
Random Energy Model (REM) \cite{Derrida1981} and the Sherrington--Kirkpatrick model :  it is
analytically tractable, yet it carries a genuine local correlation structure between
configurations, absent from the REM. Over the last two decades the mathematical theory of
the BRW has developed considerably, we refer to Shi \cite{shi2015} and Zeitouni \cite{zeitouni20notes} for surveys. In particular,
Jagannath \cite{jagannath2016} studied the binary Gaussian BRW from the statistical-physics point
of view and proved that it exhibits a one-step replica symmetry breaking (1-RSB) behaviour at
low temperature, obtaining a precise form for its overlap distribution and showing that the
associated Gibbs measure satisfies the Ghirlanda--Guerra identities. In this sense, the binary
Gaussian BRW belongs to the same universality class as the REM. The critical point itself
has been studied separately by Pain \cite{pain2018}, who described the near-critical
Gibbs measure of the BRW as $\beta \downarrow \beta_c$ --- a regime
complementary to the strictly low-temperature phase $\beta > \beta_c(h)$
considered here. Mallein \cite{Mallein2018} obtained the limiting law of the overlap between two configurations
drawn according to the low-temperature Gibbs measure, as consequence of the joint convergence of
the extremal process and its genealogical structure, using in particular the description of
the extremal process obtained by Madaule \cite{madaule2017}. The overlap statistics of the closely
related branching Brownian motion (BBM) --- a continuous-time, log-correlated relative of the
BRW --- have also attracted recent attention : see Bonnefont \cite{bonnefont22} for the two-temperature
overlap distribution of the BBM, Bonnefont, Pain and Zindy \cite{bonnefontpainzindy25} for a comparison, in the
supercritical phase, between the overlap distributions and temperature susceptibilities of the
BBM and of the REM, and Chataignier and Pain \cite{chataignierpain24} for sharp asymptotics of the overlap
distribution of the BBM throughout the whole subcritical (high-temperature) phase.
The effect of an external magnetic field on such log-correlated disordered systems has been
studied along two lines, depending on whether the field is
{\it uniform} (deterministic) or {\it random}.

\medskip

\noindent\textbf{Uniform external field.} For the REM, the free energy in a uniform magnetic
field was computed heuristically by Derrida \cite[Section 9]{derrida80}, and for the GREM
by Derrida and Gardner \cite{derridagardner86b}, who determined the magnetization profile $q(x)$ of the model.
A rigorous derivation for the GREM was given by Bovier and Klimovsky \cite{bovierklimovsky2008}, who proved a
concentration result for the fluctuations of the partition function and, in particular, showed
that the coarse-grained parts of the system that survive in the thermodynamic limit carry a
well-defined {\it optimal magnetization}, determined by the strength of the external field.
The proof strategy of \cite{bovierklimovsky2008} --- reducing the analysis of the full partition function to
that of {\it partition functions restricted to a prescribed magnetization}, controlled
uniformly via a large-deviation (Chernoff-type) estimate, and matched against a global lower
bound obtained by a concentration of measure argument --- is the strategy that we adapt below,
in Section \ref{sec:magnet}, to establish Theorem \ref{thm:theorem1} for the BRW.

\medskip

\noindent\textbf{Random external field.} For the REM, the analogous problem with a
{\it random} magnetic field was introduced, on the physics side, by de Oliveira Filho, da Costa
and Yokoi \cite{OCY06}, who computed the free energy by the replica method. A fully rigorous
treatment, including the fluctuations of the ground states and a proof of one-step replica
symmetry breaking, was subsequently given by Arguin and Kistler \cite{arguinkistler14}. The corresponding
question for the (two-level) GREM with a random magnetic field was solved by
Persechino \cite{Persechino18} in his doctoral thesis, and published jointly with Arguin
in \cite{arguinpersechino19}, where the maximum, the entropy, and the free energy of the model are determined.
The proofs again rely on a large-deviation reduction to a GREM restricted to configurations of
prescribed magnetization, in the spirit of \cite{bovierklimovsky2008}. We note that \cite{Persechino18} explicitly
observes that this method is general enough to extend to the $k$-level GREM with a random
magnetic field, and to the branching random walk coupled with a random magnetic field. The
present paper may be seen as carrying out this extension to the BRW in the complementary,
technically simpler case of a {\it uniform} field, while additionally identifying the exact
{\it optimal} magnetization and relating the resulting picture to the overlap results of
Mallein \cite{Mallein2018} and Jagannath \cite{jagannath2016} for the {\it field-free} BRW. 

\medskip
\medskip

Summarizing, the key observation of this paper is that, despite the addition of a uniform
external magnetic field, the resulting model can still be seen as a BRW, for which the
displacements with respect to the parent particle remain independent, but are no longer
identically distributed. Combining recent \cite{Mallein2018} but also older \cite{biggins76,chauvinrouault97}
results on {\it general} BRW then allows for a rather complete understanding of this model from the
statistical physics point of view : the ground state and the free energy, the existence of a
1-RSB regime, the overlap distribution, and Poisson--Dirichlet statistics for the Gibbs weights
at low temperature. Building on this picture, we prove a strong concentration result for the
magnetization (Theorem \ref{thm:theorem1}), whose proof follows the large-deviation, magnetization-slicing
strategy of Bovier and Klimovsky \cite{bovierklimovsky2008}, adapted here to the BRW-correlated setting. 
It is tempting to deduce the distribution of the classical Ising overlap between two independently sampled
configurations directly from Theorem \ref{thm:theorem1}, simply by applying the concentration result to each replica
separately. We show in Section \ref{sec:overlap}  that this argument is not sufficient : the joint law of
two replicas is not determined by their marginals (see Remark \ref{rem:counter-ex}), and a genuinely
two-replica argument is required instead (see Remark \ref{rem:new2Dtech}).

 The particular two-replica implementation developed here appears to be new in this setting, and specific to this tree-indexed model. But the idea behind it is classical : a joint {\it type} that is not simply the product of its two
marginals is exponentially rare. This fact is well known in information theory, where it is called the subadditivity of entropy,
a basic result in the {\it method of types} (see Cover and Thomas \cite[Chapter 11]{CoverThomas06}). Writing this connection down explicitly --- see
Lemma \ref{lem:subadditivity} in the Appendix --- makes the proof of Theorem \ref{thm:theorem2} shorter,
and gives a simple explanation for why two unrelated replicas still share a strictly positive overlap $(m^*)^2$.
Using this two-replica large-deviation argument, we obtain the distribution of a second {\it hypercube-type} overlap (Theorem  \ref{thm:theorem2}), which we compare with
the results of Arguin and Kistler \cite{arguinkistler14} for the REM in a {\it random} external field.

\medskip

\subsection{The model}
\label{subsec:model}

We now describe the model precisely and isolate the elementary, purely algebraic identity on
which the whole paper relies.

\medskip

\noindent\textbf{The branching random walk.}
We consider a one-dimensional Gaussian binary discrete-time branching random walk (BRW) on the
real line $\mathbb{R}$. At the beginning, there is a single particle located at the origin $0$.
It has two children, which form the first generation and are positioned according to two i.i.d.\
centered Gaussian random variables with variance one. Each particle of the first generation
independently gives birth to two new particles, positioned (with respect to their birth place)
according to the same procedure; they form the second generation, and so on. For every $n\ge 1$,
each particle at generation $n$ produces two new particles, independently of one another and of
everything up to the $n$-th generation.

The particles of the branching random walk form a binary tree, denoted by $\mathcal{T}$. We call
$\varnothing$ the root. For every vertex $\sigma\in\mathcal{T}$, we denote by $|\sigma|$ its
generation (so that $|\varnothing|=0$). We write $\sigma\succ v$ if $\sigma$ is a descendant of
$v$ (equivalently, if $v$ is an ancestor of $\sigma$) in $\mathcal{T}$, so that $\mathcal{T}$
encodes the genealogy of the branching random walk. We use the lexicographic order with alphabet
$\{-1,1\}$ --- rather than the alphabet $\{0,1\}$ more commonly found in the BRW
literature --- in order to match the spin-glass notation and the definition of the overlap : any $\sigma\in\mathcal{T}$ with $|\sigma|=n$ is written
$\sigma=\sigma_1\ldots\sigma_n$, with $\sigma_i\in\{-1,1\}$ for all $i\in\{1,\ldots,n\}$. We
write $\Sigma_n:=\{-1,1\}^n$ for the set of all vertices at generation $n$. With this
convention, every vertex $\sigma\in\Sigma_n$ is simultaneously a leaf of $\mathcal T$ at
generation $n$ and an Ising spin configuration $(\sigma_1,\ldots,\sigma_n)$ on $n$ sites. This
double reading is precisely the one used throughout the REM/GREM literature discussed
above, and it is the reason for departing from the $\{0,1\}$-labelling.

\medskip

\noindent {\it \underline{Notation}.} We record here, in one place, additional notations used
throughout the paper : for $\sigma\in\Sigma_n$, $\sigma|_i :=
\sigma_1\cdots\sigma_i\in\Sigma_i$ denotes the restriction (ancestor) of $\sigma$ at
generation $i\le n$, so that $\sigma|_n=\sigma$; and $\sigma\wedge\sigma' \in
\bigcup_{i\le n}\Sigma_i$ denotes the most recent common ancestor of $\sigma,\sigma'\in
\Sigma_n$, i.e.\ the longest common prefix of $\sigma$ and $\sigma'$ read as words on
$\{-1,1\}$.

\medskip

To define the BRW properly, we associate to all vertices/particles/spin configurations (except
the root) i.i.d.\ standard Gaussian random variables $(U(\sigma))_{\sigma\in\mathcal{T}\setminus
\varnothing}$. The position of the particle $\sigma$ is then given by $X_0(\varnothing)=0$, or, for
$\sigma\ne\varnothing$, by summing the Gaussian random variables along the geodesic between
$\varnothing$ and $\sigma$ :
\[
X_n(\sigma):=\sum_{1\le i\le n}U(\sigma|_i), \qquad \forall\,\sigma\in\Sigma_n,\ \forall\,n\ge1.
\]

\begin{rem}[Ultrametric, log-correlated disorder]
Equivalently, $(X_n(\sigma))_{\sigma\in\Sigma_n}$ is a centered Gaussian field on $\Sigma_n$
with covariance structure given by :
\[
\mathrm{Cov}\left(X_n(\sigma),X_n(\sigma')\right)=|\sigma\wedge\sigma'|,
\qquad \forall \, \sigma,\sigma'\in\Sigma_n, \ \forall\,n\ge1.
\]
This {\it ultrametric} covariance structure --- degenerating to the i.i.d.\
case $\mathrm{Cov}(X_n(\sigma),X_n(\sigma'))=0$, $\sigma\ne\sigma'$, of the REM when the tree has
a single generation --- is exactly what the BRW shares with the GREM and with the branching
Brownian motion, and is the source of the strong local correlations mentioned in
Subsection \ref{subsec:literature}.
\end{rem}

\medskip

\noindent\textbf{The Hamiltonian with a uniform magnetic field.}
We now couple this disordered system with a uniform external magnetic field of strength $h>0$ :
we introduce
\[
H_n(\sigma,h):=X_n(\sigma)+h\sum_{1\le i\le n}\sigma_i, \qquad \forall\,\sigma\in\Sigma_n,\ \forall\,n\ge1.
\]
In statistical-physics language, $H_n(\sigma,h)$ is the energy of the spin configuration
$\sigma\in\Sigma_n=\{-1,1\}^n$, obtained by adding to the quenched disorder $X_n(\sigma)$ a
linear coupling of each spin $\sigma_i$ to a common external field $h$. This is exactly the
coupling considered for the REM by Derrida \cite{derrida80} and, in the random-field case, by
de Oliveira Filho, da Costa and Yokoi \cite{OCY06} and Arguin and Kistler \cite{arguinkistler14}, and for the
GREM by Derrida and Gardner \cite{derridagardner86b} and Bovier and Klimovsky \cite{bovierklimovsky2008}. 
Here it is transplanted onto the correlated Gaussian field $(X_n(\sigma))_{\sigma\in\Sigma_n}$ described above.

\medskip

\noindent\textbf{The key algebraic identity.}
A central --- and elementary --- observation of this paper is the following rewriting of
$H_n(\sigma,h)$. Since $\sigma_i$ is determined by which of the two children is visited at step
$i$, independently of the disorder accumulated so far, we may regroup the sum defining
$H_n(\sigma,h)$ term by term along the ancestral path of $\sigma$ :
\[
H_n(\sigma,h)=\sum_{1\le i\le n}\left(U(\sigma|_i)+h\,\sigma_i\right),
\]
and the law of $U(\sigma|_i)+h\sigma_i$ is manifestly $\mathcal{N}(h\sigma_i,1)$.
Consequently, $(H_n(\sigma,h))_{\sigma\in\Sigma_n,\,n\ge1}$ is itself a branching random walk,
now with reproduction law
\[
\delta_{\mathcal{N}(h,1)}+\delta_{\mathcal{N}(-h,1)},
\]
where the two Gaussian random variables $\mathcal{N}(h,1)$ and $\mathcal{N}(-h,1)$ are
independent : at each branching event, one child is displaced according to $\mathcal{N}(h,1)$
(the child labelled $+1$) and the other according to $\mathcal{N}(-h,1)$ (the child labelled
$-1$), instead of the common law $\mathcal{N}(0,1)$ of the {\it field-free} model. This is illustrated
in Figure \ref{fig:reproduction}. In particular, the reproduction law is no longer identical
across children, but the two displacements remain independent of each other and of the rest of
the tree, so that every general result available in the BRW literature --- in particular
Biggins \cite{biggins76}, Chauvin and Rouault \cite{chauvinrouault97} and Mallein \cite{Mallein2018} --- applies
directly to $H_n(\cdot,h)$.

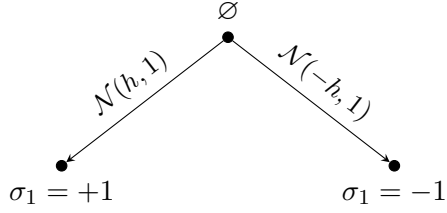
\begin{figure}[h]
\centering
\begin{tikzpicture}[>=stealth,scale=1]
\node[circle,draw,fill=black,inner sep=1.4pt,label=above:{$\varnothing$}] (r) at (0,0) {};
\node[circle,draw,fill=black,inner sep=1.4pt,label=below:{$\sigma_1=+1$}] (p) at (-2.2,-1.7) {};
\node[circle,draw,fill=black,inner sep=1.4pt,label=below:{$\sigma_1=-1$}] (m) at (2.2,-1.7) {};
\draw[->] (r) -- (p) node[midway,above,sloped,font=\small] {$\mathcal{N}(h,1)$};
\draw[->] (r) -- (m) node[midway,above,sloped,font=\small] {$\mathcal{N}(-h,1)$};
\end{tikzpicture}
\caption{The reproduction law of the {\it tilted} BRW $H_n(\cdot,h)$ : the child labelled $+1$ is
displaced, relative to its parent, according to $\mathcal N(h,1)$, the child labelled $-1$
according to $\mathcal N(-h,1)$, the two displacements being independent. Every particle
reproduces independently according to this same law.}
\label{fig:reproduction}
\end{figure}

\begin{rem}[Two limiting regimes]
Two heuristic regimes help to situate the model. As $h\downarrow0$, the reproduction law
$\delta_{\mathcal N(h,1)}+\delta_{\mathcal N(-h,1)}$ degenerates to $\delta_{\mathcal N(0,1)}+\delta_{\mathcal N(0,1)}$, where the two $\mathcal N(0,1)$'s are independent random variables, and one recovers the {\it field-free} BRW of Derrida and Spohn \cite{derridaspohn88}, studied by
Jagannath \cite{jagannath2016}. As $h\to\infty$, the {\it optimal} magnetization $m^*(h)\in(0,1)$, defined in Equation \eqref{eq:m*}, satisfies $m^*(h)\to1$ and the Gibbs measure becomes asymptotically field-aligned, in the sense that the density of negative spins vanishes. The interesting regime, on
which this paper focuses, is the intermediate one : $0<h<\infty$, in which disorder and field
compete on the same linear scale $n$, and where --- as we shall see --- a nontrivial {\it optimal}
magnetization  $m^*(h)\in(0,1)$ emerges from this competition.
\end{rem}

\medskip

The remainder of this paper is organized as follows. Section \ref{sec:results} collects, in
Theorems \ref{thm:theoremA}--\ref{thm:theoremC}, the results on the ground state, the free energy, and the {\it genealogical} overlap distribution
that follow from the general theory of branching random walks (Biggins \cite{biggins76}, Chauvin and Rouault \cite{chauvinrouault97}, Mallein \cite{Mallein2018}),
 and states our two new results : a concentration result for the magnetization under the low-temperature Gibbs measure
(Theorem \ref{thm:theorem1}), and the distribution of a second {\it hypercube-type} overlap
(Theorem \ref{thm:theorem2}), compared with the results of Mallein \cite{Mallein2018} and Arguin and Kistler \cite{arguinkistler14}. Section \ref{sec:magnet} is devoted to
the proof of Theorem \ref{thm:theorem1} and Section \ref{sec:overlap} to that of Theorem \ref{thm:theorem2}. 
An appendix collects the many-to-one lemma (Appendix \ref{subsec:appendix-1}), the
information-theoretic fact needed, namely the subadditivity of Shannon entropy
(Appendix \ref{app:subadditivity}), and the proofs of Theorems \ref{thm:theoremA}--\ref{thm:theoremC}, together with the auxiliary Lemma \ref{lem:Psi-h}
locating the critical inverse temperature $\beta_c(h)$ (Appendix \ref{subsec:proofABC}).

\medskip

\section{Results}
\label{sec:results}

\subsection{From the BRW literature}
\label{subsec:brw-literature}

Recall from Subsection \ref{subsec:model} that $(H_n(\sigma,h))_{\sigma\in\Sigma_n,\,n\ge1}$ is a
general branching random walk with reproduction law
$\delta_{\mathcal N(h,1)}+\delta_{\mathcal N(-h,1)}$. The three results collected in this
subsection --- the ground state, the free energy, and the overlap distribution --- are direct
consequences of the general theory of BRW, and we simply record them in the notation of our
model, referring to Appendix  \ref{subsec:proofABC} for the ``proofs''.
In a pioneering work, Biggins \cite{biggins76} determined the first-order behaviour of the maximum
for a general branching random walk. Applied to   $H_n(\cdot,h)$, his result yields the ground
state of the model.


\medskip

\begin{thmA}[Ground state - Biggins \cite{biggins76}]
\label{thm:theoremA}
The ground state is given by
\begin{equation}
\label{eq:gammamax}
\lim_{n \to \infty} \frac{1}{n} \, \max_{ \sigma \in \Sigma_n} H_n(\sigma,h)= \gamma_{\max}(h) :=\beta_c + h \tanh(\beta_c h), \quad \text{a.s. and in } L^1,
\end{equation}
where $\psi_h(t):=  \frac{t^2}{2} + \log 2 + \log (\cosh(ht))$, for all $t \in \R$, and  $\beta_c=\beta_c(h)>0$ is the unique positive real number satisfying
\begin{equation}
\label{eq:beta_c}
\beta_c \psi'_h(\beta_c)- \psi_h(\beta_c)=0,
\end{equation}
(see Lemma \ref{lem:Psi-h} for existence and uniqueness of $\beta_c$).
\end{thmA}

\medskip


\begin{rem}
\label{rem:groundstate-split}
Equation  \eqref{eq:gammamax} splits the maximal energy density into two contributions : the
energy density $\beta_c$ of a {\it constrained} ({\it field-free}) BRW, and the energy density
$h\tanh(\beta_c h)$ gained from aligning the spins along the ancestral path with the field. As
we shall see in Proposition \ref{prop:var}, $\beta_c$ is itself the entropy cost of maintaining the empirical
magnetization at its optimal value $m^*(h):=\tanh(\beta_ch)$, so that
$\gamma_{\max}(h)=\beta_c+h\,m^*(h)$ already anticipates the central role played by $m^*(h)$
throughout this note.
\end{rem}


The associated Gibbs measure and free energy, at inverse temperature $\beta>0$ and with a
uniform external field $h>0$, are respectively defined by
$$
\cG_{\beta,h,n}(\sigma)
:= \frac{\ee^{\beta H_n(\sigma,h)}}{Z_{\beta,h,n}},  \quad  \forall \, \sigma \in \Sigma_n, \qquad f_n(\beta,h) :=  \frac{\log Z_{\beta,h,n}}{n}, \qquad  \forall \,  n \ge 1,
$$
where 
$$
Z_{\beta,h,n}:= \sum_{\sigma \in \Sigma_n} \ee^{\beta H_n(\sigma,h)}
$$
is the associated partition function. As for the REM and the GREM, the model exhibits a phase
transition in the large $n$ behavior of the free energy, located precisely at the critical
inverse temperature $\beta_c(h)$ of Theorem \ref{thm:theoremA}.


\medskip

\begin{thmA}[Free energy - Chauvin and Rouault \cite{chauvinrouault97}]
\label{thm:theoremB}
For all $\beta,h>0$, the limiting free energy exists and is given by 
\begin{align*}
\lim_{n\to\infty} f_n(\beta,h) 
= f(\beta,h) :=  \left\{
\begin{array}{ll}
\log 2+  \frac{\beta^2}{2} +\log (\cosh(\beta h)), & \text{if } \beta \leq \beta_c(h), \\
 \beta \, \gamma_{\max}(h), & \text{if } \beta \geq \beta_c(h),
\end{array}
\right.
\quad \text{a.s. and in } L^1,
\end{align*}
where $\gamma_{\max}(h)$ is defined by Equation \eqref{eq:gammamax} and the critical inverse temperature $\beta_c(h)>0$ by Equation \eqref{eq:beta_c}.
\end{thmA}

\medskip


Following Jagannath \cite{jagannath2016}, let us introduce the {\it genealogical} overlap I
between two configurations $\sigma,\sigma'\in\Sigma_n$, $n\ge1$, defined by
\[
q_n(\sigma,\sigma'):=
\begin{cases}
1, & \text{if } \sigma=\sigma',\\[2pt]
\dfrac1n\left(\min\{1\le i\le n:\sigma_i\ne\sigma_i'\}-1\right), & \text{otherwise}.
\end{cases}
\]
In words, $q_n(\sigma,\sigma')$ is the (rescaled) generation of the most recent common ancestor
of $\sigma$ and $\sigma'$: it measures how far, genealogically, the two configurations remain
identical before branching apart. Mallein \cite{Mallein2018} obtained the limiting law of this
overlap between two configurations drawn independently according to the low-temperature Gibbs
measure, by studying the joint convergence of the extremal process together with its
genealogical structure, building in particular on the description of the extremal process
obtained by Madaule \cite{madaule2017}.


\medskip

\begin{thmA}[Overlap I distribution -  Mallein \cite{Mallein2018}]
\label{thm:theoremC} 
Let $h>0$ and  $\beta>\beta_c(h)>0$, where $\beta_c(h)$ is defined by Equation \eqref{eq:beta_c}. Then, 
\begin{align*}
  \lim_{n \to +\infty} \E \left[ \cG_{\beta,h,n}^{\otimes 2}\{ q_n(\sigma,\sigma') \in \, \cdot \, \} \right] =  \frac{\beta_c(h)}{\beta} \delta_0 + \left(1 -  \frac{\beta_c(h)}{\beta}\right) \delta_1.
  \end{align*}
\end{thmA}

\medskip


\begin{rem}
\label{rem:RS-1RSB}
Chauvin and Rouault \cite{chauvinrouault97} proved that the limit above is $\delta_0$ when $\beta\le\beta_c(h)$. 
In the language of the Replica theory of spin glasses, the model is therefore
{\it Replica Symmetric} (RS) at high temperature $\beta\le\beta_c(h)$ and exhibits
{\it one-step Replica Symmetry Breaking} (1-RSB) at low temperature $\beta>\beta_c(h)$, with
transition exactly at $\beta_c(h)$.
\end{rem}


\medskip

\begin{rem}
\label{rem:PD}
Mallein \cite{Mallein2018} also proved that, for every $\beta>\beta_c(h)$, the Gibbs weights
$(\mathcal G_{\beta,h,n}(\sigma))_{\sigma\in\Sigma_n}$, ranked in decreasing order, converge as
$n\to\infty$ to a Poisson--Dirichlet random variable with parameter $\beta_c(h)/\beta$. Together
with Remark \ref{rem:RS-1RSB}, this places the model, at low temperature, in the same
universality class as the REM and the GREM with a uniform external magnetic field.
\end{rem}

\subsection{New results}
\label{subsec:new-results}

Theorem \ref{thm:theoremC} describes the {\it genealogical} relationship between two independently sampled
configurations, but it says nothing about the spins they actually carry. To go further, let us
introduce the empirical magnetization of a configuration $\sigma\in\Sigma_n$, $n\ge1$, and the
{\it optimal} magnetization (depending on $h$) :
\begin{equation}
\label{eq:m*}
y_n(\sigma):=\frac1n\sum_{1\le i\le n}\sigma_i,
\qquad\qquad
m^*=m^*(h):=\tanh(\beta_c(h)\,h).
\end{equation}
The quantity $m^*(h)$ already appeared implicitly in Remark  \ref{rem:groundstate-split} : it is
the magnetization carried, to leading order, by the ground state of Theorem \ref{thm:theoremA}. Our first new
result shows that this is not merely a feature of the extremal configurations, but a bulk
property of the whole low-temperature Gibbs measure, in more concrete terms : a configuration sampled according to
$\mathcal G_{\beta,h,n}$ has an empirical
magnetization within any prescribed distance $\eta>0$ of $m^*(h)$, except on
an event of $\mathcal G_{\beta,h,n}$-probability exponentially small in
$n$, for {\it every} $\beta>\beta_c(h)$ and --- this is the genuinely new statement, extending beyond the
single point $\beta=\beta_c(h)$ implicit in Remark \ref{rem:groundstate-split} --- not only at $\beta=\beta_c(h)$. This
is the counterpart, for the correlated BRW, of the concentration of magnetization result
obtained by Bovier and Klimovsky \cite{bovierklimovsky2008} for the GREM with a uniform magnetic field ---
recall the discussion in Subsection \ref{subsec:literature} --- and its proof, given in Section \ref{sec:magnet}, follows the same
large-deviation strategy : the Gibbs measure is sliced according to the value of $y_n(\sigma)$,
each slice is controlled by a uniform Chernoff bound, and the bound is matched against a global
lower bound on $Z_{\beta,h,n}$ obtained by Gaussian concentration of measure. Applying this result independently to the two replicas gives joint
concentration of both magnetizations near $(m^*,m^*)$ (see Corollary \ref{cor:marginals} below) 
--- though, as we discuss in Subsection \ref{subsec:coro_immediate} and Remark \ref{rem:counter-ex}, this one-replica
statement alone is {\it not} sufficient to control the  {\it hypercube-type}  overlap itself, and that a two-replica large-deviation argument
is required. This argument, developed in Section \ref{sec:overlap}, also gives a useful two-replica large-deviation method (see Remark \ref{rem:new2Dtech}).


\medskip

\begin{thm}[Concentration of the magnetization]
\label{thm:theorem1} 
Fix $h>0$ and let $\beta > \beta_c(h)>0$, where $ \beta_c(h)$ is defined by Equation \eqref{eq:beta_c}. Then, for any $\eta>0$,
\begin{align*}
  \lim_{n \to +\infty} \E \left[ \cG_{\beta,h,n}\{ \vert y_n(\sigma) - m^* \vert > \eta \} \right] =  0.
  \end{align*}
More precisely, there exist $c=c(\beta,h,\eta)>0$ and $C=C(\beta,h,\eta)<+\infty$, such that 
\begin{equation}
\label{eq:quant}
 \E \left[ \cG_{\beta,h,n}\{ \vert y_n(\sigma) - m^* \vert > \eta \} \right] \le C \ee^{-c n}, \qquad \forall \, n \geq 1,
  \end{equation}
  which implies
  \begin{align*}
\lim_{n\to\infty}\mathcal G_{\beta,h,n}\{\vert y_n(\sigma)-m^*\vert>\eta\}=0, \qquad \textrm{a.s.}
  \end{align*}
\end{thm}

\medskip

In order to compare this picture with the one obtained by Arguin and Kistler \cite{arguinkistler14} for the
REM in a {\it random} magnetic field, it is natural to consider, alongside the {\it genealogical}
overlap  I, the more classical Ising {\it hypercube-type} overlap. For all $n\ge1$ and $\sigma,\sigma'\in
\Sigma_n$, let
\[
r_n(\sigma,\sigma'):=\frac1n\sum_{1\le i\le n}\sigma_i\sigma_i'\in[-1,1]
\]
and denote this {\it overlap II}. Combining the genealogical description of  Theorem \ref{thm:theoremC} with the two-replica
large-deviation technique of Section \ref{sec:overlap} then yields the following distribution for
overlap II.


\medskip

\begin{thm}[Overlap II distribution]
\label{thm:theorem2}  
Fix $h>0$ and let $\beta > \beta_c(h)>0$, where $ \beta_c(h)$ is defined by Equation \eqref{eq:beta_c} and $m^*=m^*(h)$ by Equation \eqref{eq:m*}. Then
\begin{align*}
  \lim_{n \to +\infty} \E \left[ \cG_{\beta,h,n}^{\otimes 2}\{ r_n(\sigma,\sigma') \in \, \cdot \, \} \right] =  \frac{\beta_c(h)}{\beta} \delta_{(m^*)^2} + \left(1 -  \frac{\beta_c(h)}{\beta}\right) \delta_1.
  \end{align*}
\end{thm}

\medskip


\begin{rem}[Two overlaps, and how they complement each other]
\label{rem:two-overlaps}
Theorem \ref{thm:theorem2} calls for two comments.

\medskip

{\it (a) Consistency with Arguin--Kistler and with Bovier--Klimovsky.} The overlap II
considered here is exactly the Ising overlap studied by Arguin and Kistler  \cite[Corollary 4]{arguinkistler14}
for the REM in a random magnetic field, and its limiting law has the same qualitative shape ---
an atom at $1$, reflecting identical configurations, and an atom at a strictly positive value
below $1$, reflecting the residual alignment of two {\it independent} configurations under the
field. The  {\it genealogical} overlap I of Theorem \ref{thm:theoremC}, on the other hand, has no counterpart in the REM
(all configurations being equidistant there). It is specific to the correlated, tree-indexed
setting of the BRW and of the GREM, and is closer in spirit to the magnetization-profile results
of Bovier and Klimovsky \cite{bovierklimovsky2008}, which concern the GREM rather than the overlap itself. Theorem \ref{thm:theorem2}
connects these two pictures : it transfers the REM-type overlap computation
of \cite{arguinkistler14} into the correlated setting, using the genealogical structure available there but
absent from the REM.

\medskip

{\it (b) Complementing Theorem \ref{thm:theoremC} of Mallein.} Theorem \ref{thm:theoremC} only tells us that two independently sampled
configurations are, asymptotically, either genealogically identical (with probability
$1-\beta_c(h)/\beta$) or genealogically unrelated from the root onward (with probability
$\beta_c(h)/\beta$). But {\it it says nothing about how correlated their spins actually are} in the second
case. One might expect that two genealogically unrelated configurations behave, at the
level of the spins, as two independent uniform points of $\{-1,1\}^n$, so that
$r_n(\sigma,\sigma')\to0$. Theorem \ref{thm:theorem2} rules this out : two genealogically unrelated
configurations satisfy $r_n(\sigma,\sigma')\to (m^*(h))^2>0$, as soon as
$h>0$ --- see Remark \ref{rem:new2Dtech} below for why this cannot be deduced
directly from Theorem \ref{thm:theorem1}. In this sense, Theorem \ref{thm:theorem2}
both completes Mallein's overlap I result with a quantitative, spin-level statement, and gives
the exact non-trivial value --- namely $(m^*(h))^2=\tanh^2(\beta_c(h)h)$.
\end{rem}

\medskip

\begin{rem}[A new two-replica technique]
\label{rem:new2Dtech}
The proof of Theorem \ref{thm:theorem2}, given in Section \ref{sec:overlap}, is {\it not} a direct consequence of
Theorem \ref{thm:theorem1} : as Remark \ref{rem:counter-ex} shows by an elementary counterexample, concentration of each
replica's magnetization does not determine their joint correlation. We instead develop
a self-contained two-replica large-deviation technique, extending the one-replica
magnetization-slicing strategy of Section \ref{sec:magnet} to the {\it joint empirical type} of a pair of
configurations, via a two-dimensional rate function $\Xi_{\beta,h}$ built from a simple
polar-coordinate symmetrization identity $g_\beta^{(2)}(S) = 2g_\beta(S/2)$, see Equation
\eqref{eq:g2-identity}. 

Each ingredient is classical, and we make no claim to the contrary.
The counting of pairs by their joint empirical type, together with the
polynomial bound on the number of types that makes the union bound free, is
the {\it method of types} of Shannon theory (see Csisz\'{a}r and K\"orner \cite{Csiszar2011},
 see also \cite[Chapter 11]{CoverThomas06}). That a non-product joint type is
exponentially penalized, and that the optimal type under linear constraints
is the measure of maximal entropy compatible with them --- here the product of two
$m^*$-magnetized marginals --- is the {\it Gibbs conditioning principle},
or conditional limit theorem, of Csisz\'{a}r \cite{Csiszar84}. Finally, restricted first-moment estimates of this kind go back,
for branching random walks, to Biggins \cite{biggins77a}, and, for spin glasses, to the
microcanonical computation of Derrida \cite{Derrida1981} for the REM, of which the analyses
of Bovier and Klimovsky \cite{bovierklimovsky2008} and Arguin and Kistler \cite{arguinkistler14} under a field are the
direct descendants.

What we believe to be new is not any one of these ingredients but their
combination in the correlated, tree-indexed setting, together with the
identity in Equation \eqref{eq:g2-identity}. Three points are specific to the present model.
First, the joint type is carried by the part of the tree {\it beyond the
most recent common ancestor}, which forces a separate uniform control of the
head (Lemma \ref{lem:head}) that has no counterpart in the REM. Second, the
polar-coordinate identity $g^{(2)}_\beta(S)=2g_\beta(S/2)$ reduces the
two-replica variational problem to the very same one-dimensional function
$g_\beta$ already governing Theorem \ref{thm:theorem1}, which is what keeps Proposition \ref{prop:variational} 
short. Third, the outcome is matched against the genealogical decomposition
of Mallein \cite{Mallein2018}, which is unavailable in the REM. By contrast, the overlap results previously
available for the field-free BRW/GREM/BBM rely on the perturbative method of Bovier
and Kurkova \cite{bovierkurkova2004-1,bovierkurkova2004-2} or on convergence of the extremal process, while the overlap under {\it random} magnetic field
result of Arguin and Kistler  \cite[Corollary 4] {arguinkistler14} for the REM crucially exploits the
randomness of the field together with the absence of a correlation structure in the
REM --- neither of which is available here.
\end{rem}

\section{Proof of Theorem \ref{thm:theorem1}}
\label{sec:magnet}

\subsection{The two-dimensional encoding}
\label{subsec:2D-encoding}

Let us first introduce the {\it binary entropy} defined by
\begin{equation}
\label{eq:binary-entropy}
s(m):=-\left(\frac{1+m}{2}\right) \, \log\left(\frac{1+m}{2}\right)-\left(\frac{1-m}{2}\right)\, \log\left(\frac{1-m}{2}\right),
\qquad \forall \, m\in[-1,1],
\end{equation}
with the convention $0\log 0=0$, so that $s$ is continuous on $[-1,1]$,
$s(\pm1)=0$, $s(0)=\log2$, $s$ is strictly concave on $(-1,1)$, strictly
decreasing on $[0,1]$ (see Figure \ref{fig:binary-entropy} below), and
\begin{equation}
\label{eq:s-prime}
s'(m)=-\arctanh(m),\qquad \forall \, m \in(-1,1).
\end{equation}

\begin{figure}[h]
\centering
\begin{tikzpicture}[>=stealth,scale=1]

  \draw[->] (-4.8,0) -- (4.8,0) node[right] {$m$};
  \draw[->] (0,-0.3) -- (0,3.3) node[right] {$s(m)$};

  \foreach \x in {-4,4}
    \draw[gray!25] (\x,-0.15) -- (\x,3.0);

 \foreach \x/\xlab in {-4/-1,4/1} {
    \draw (\x,0.06) -- (\x,-0.06);
    \node[below] at (\x,-0.06) {\small $\xlab$};
  }
  \node[below] at (0,-0.06) {\small $0$};

  \draw[dashed,gray] (0,{4*ln(2)}) -- (-4.8,{4*ln(2)});
  \node[left] at (-4.85,{4*ln(2)}) {\small $\log 2$};
  \draw[dashed,gray] (0,0) -- (0,{4*ln(2)});

  \draw[red,thick,smooth,samples=100,domain=-0.99:0.99,variable=\t]
    plot ({4*\t},{4*( -((1+\t)/2)*ln((1+\t)/2) - ((1-\t)/2)*ln((1-\t)/2) )});

\end{tikzpicture}
\caption{The  {\it binary entropy} $s(m)$ for $ m \in [-1,1]$.}
\label{fig:binary-entropy}
\end{figure}
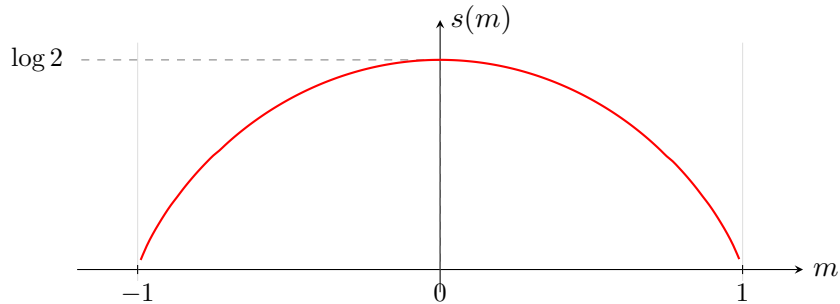

A key observation is to consider the {\it two-dimensional} BRW obtained by recording jointly the energy and the magnetization : 

\[
{\bf V}_n(\sigma):=\left(X_n(\sigma),\ \sum_{i=1}^n\sigma_i\right)\in\R^2, \qquad \forall \sigma\in\Sigma_n.
\]

Then $({\bf V}_n(\sigma))_{\sigma\in\cT}$ is a branching random walk on $\R^2$ :
each particle has exactly two children, whose displacements relative to their
parent are the two {\it independent} but {\it non-identically distributed}
random vectors $(\mathcal N(0,1),+1)$ and $(\mathcal N(0,1),-1)$. In other words the
reproduction point process is
$\delta_{(\mathcal N(0,1),+1)}+\delta_{(\mathcal N(0,1),-1)}$. Its
log-Laplace transform is

\begin{equation}
\label{eq:Lambda}
\Lambda(t,\theta):=\log\E\left[\sum_{|u|=1}\ee^{\langle(t,\theta),{\bf V}_1(u)\rangle}\right]
=\log\left(\ee^{t^2/2+\theta}+\ee^{t^2/2-\theta}\right)
=\frac{t^2}{2}+\log2+\log(\cosh\theta) ,
\end{equation}
which is finite for every $(t,\theta)\in\R^2$. Observe two important relations, which will be
used repeatedly :
\begin{equation}
\label{eq:consist}
H_n(\sigma,h)=\left\langle (1,h),{\bf V}_n(\sigma)\right\rangle \qquad\text{and}\qquad \psi_h(t)=\Lambda(t,ht).
\end{equation}
We have the following large-deviations result.

\medskip

 \begin{lem}[Rate function]
 \label{lem:rate}
Let $\Lambda^*(a,m):=\sup_{(t,\theta)\in\R^2}\{ta+\theta m-\Lambda(t,\theta)\}$
be the Legendre transform of the function $\Lambda$ defined by Equation \eqref{eq:Lambda}. Then, for all $a\in\R$ and $m\in(-1,1)$,
\begin{equation}
\label{eq:rate}
\Lambda^*(a,m)=\frac{a^2}{2}-s(m),
\end{equation}
where $s$ is the binary entropy defined in Equation \eqref{eq:binary-entropy}. The supremum is attained at $(t,\theta)=(a,\arctanh m)$.
\end{lem}

\medskip

\begin{proof}
By Equation \eqref{eq:Lambda}, the variables decouple :
\[
\Lambda^*(a,m)=\underbrace{\sup_{t\in\R}\left\{ta-\frac{t^2}{2}\right\}}_{=a^2/2 \; \mathrm{for } \; t=a}
+\underbrace{\sup_{\theta\in\R}\left\{\theta m-\log(\cosh\theta)\right\}}_{=:\ \Gamma(m)}-\log 2 .
\]
For $|m|<1$ the function $\theta\mapsto\theta m-\log(\cosh\theta)$ is strictly
concave on $\R$, tends to $-\infty$ when $\theta \to \pm \infty,$ with derivative $m-\tanh\theta$ : the supremum is attained at
$\theta=\arctanh m$, hence using the relations $\arctanh m=\frac{1}{2}\log \left(\frac{1+m}{1-m}\right)$ and $\cosh(\arctanh m)=(1-m^2)^{-1/2}$ yields
\begin{eqnarray*}
\Gamma(m)&=&m\, \arctanh m- \log(\cosh(\arctanh m)) = \frac{m}{2}\log \left(\frac{1+m}{1-m}\right) + \frac{1}{2}\log(1-m^2)
\\
&=&\frac{1+m}{2}\log(1+m)+\frac{1-m}{2}\log(1-m) = \log2-s(m).
\end{eqnarray*}
Adding the three terms gives the conclusion.
\end{proof}

Rewriting the last equation :
\[
m \arctanh m- \log(\cosh(\arctanh m)) =  \log2-s(m),
 \]
with $m = \tanh x$, yields
\begin{equation}
\label{eq:stanh}
s(\tanh x)=\log 2+\log(\cosh x)-x\tanh x,\qquad \forall \, x\in\R,
\end{equation}
which will be convenient in Subsection \ref{subsec:var_prob}.

\medskip

Let us now introduce
\[
M_n:=\left\{-1+\frac{2k}{n}:0\le k\le n\right\}, \qquad \forall n \ge 1,
\]
the {\it admissible} values of the magnetization $y_n(\sigma)$, 
\[
\Sigma_n(m):=\{\sigma\in\Sigma_n:y_n(\sigma)=m\},  \qquad \forall n \ge 1, \, \forall \, m \in M_n,
\]
the set of configurations with magnetization $m$ and, for all $\beta,h>0,$
\[
Z_{\beta,h,n}(m):=\sum_{\sigma\in\Sigma_n(m)}\ee^{\beta H_n(\sigma,h)},  \qquad \forall n \ge 1, \, \forall \, m \in M_n,
\]
 the {\it restricted} partition function. The key point is to observe that
\[
\mathcal G_{\beta,h,n}\{|y_n(\sigma)-m^*|>\eta\}
=\frac{\sum_{m\in M_n,\,|m-m^*|>\eta}Z_{\beta,h,n}(m)}{Z_{\beta,h,n}},
\]
and that $\#M_n=n+1$ grows only polynomially. Therefore it suffices to combine a good {\it upper} bound on $\frac1n\log Z_{\beta,h,n}(m)$ valid {\it simultaneously for all} $m\in M_n$ and a matching  {\it lower} bound on $\frac1n\log Z_{\beta,h,n}$.

\subsection{A uniform Chernoff bound for restricted high points}

For $n \geq 1$, $m\in M_n$ and $a\in\R$, define the number of high points with
{\it prescribed magnetization} :
\[
N_n(a,m):=\#\left\{\sigma\in\Sigma_n(m)\ :\ X_n(\sigma)\ge an\right\}.
\]

\medskip

\begin{lem}[First moment]
\label{lem:firstmom}
For every $n\ge1$, every $m\in M_n$ and every $a\ge0$,
\begin{equation*}
\E\left[N_n(a,m)\right]\le
\ee^{-n\Lambda^*(a,m)}= \exp\left\{n\left(s(m)-\frac{a^2}{2}\right)\right\}.
\end{equation*}
\end{lem}

\medskip

\begin{proof}
We first treat the boundary cases $m=\pm1$, for which the
optimizer $\theta=\arctanh(m)$ appearing in Lemma \ref{lem:rate} is not defined, and Lemma \ref{lem:rate} covers only
$m\in(-1,1)$. For $m=1$ (resp. $m=-1$), the layer $\Sigma_n(m)$ reduces to the single vertex
$\sigma=(1,\dots,1)$ (resp.\ $\sigma=(-1,\dots,-1)$), so that $N_n(a,m)\in\{0,1\}$ almost surely  and,
since $X_n(\sigma)\sim\mathcal N(0,n)$, the classical Gaussian tail bound gives, for
every $a\ge0$,
\[
\mathbb E[N_n(a,\pm1)]=\mathbb P(X_n(\sigma)\ge an)\le e^{-na^2/2}=\exp\left\{n\left(s(\pm1)-\frac{a^2}{2}\right)\right\},
\]
using $s(\pm1)=0$. This is exactly the upper bound claimed in Lemma \ref{lem:firstmom}. We may therefore assume
from now on that $m\in(-1,1)$.
For $t\ge0$ and $\theta\in\R$, one has

\begin{eqnarray*}
N_n(a,m) & \le &  \sum_{\sigma\in\Sigma_n(m)}
\ee^{t(X_n(\sigma)-an)}=  \sum_{\sigma\in\Sigma_n(m)}
\ee^{t(X_n(\sigma)-an)+\theta(\sum_i\sigma_i-mn)}  \\
 & \le &  \sum_{\sigma\in\Sigma_n}
\ee^{t(X_n(\sigma)-an)+\theta(\sum_i\sigma_i-mn)}.
\end{eqnarray*}
 Taking expectations and using the many-to-one lemma (see
Lemma  \ref{lem:m21} in the Appendix) implies
 \[
\E[N_n(a,m)]\le \ee^{-n(ta+\theta m-\Lambda(t,\theta))}.
\]
 Optimizing over
$(t,\theta)$, which by Lemma  \ref{lem:rate} is legitimate with the admissible
choice $t=a\ge0$, $\theta=\arctanh m$, concludes the proof. 
\end{proof}

Set
\[
\bar a(m):=\sqrt{2s(m)}\in[0,\sqrt{2\log 2}], \qquad \forall \,  m\in[-1,1],
\]
which, by Lemma \ref{lem:firstmom}, is the largest energy density carried by the
layer $\Sigma_n(m)$ : for $a>\bar a(m)$ the first moment is exponentially
small, so the layer has no such particle with high probability. See Figure \ref{fig:abar} for the graph  of the function $\bar a(\cdot)$ on $[-1,1]$. We now
convert this into a bound on the restricted partition function. Define, for
$\beta>0$ and $m\in[-1,1]$,
\begin{equation}
\label{eq:Phi}
\Phi_h(\beta,m):=\beta hm+\max_{0\leq a \leq \bar a(m)}
\left\{s(m)-\frac{a^2}{2}+\beta a\right\}
=\beta hm+
\begin{cases}
s(m)+\dfrac{\beta^2}{2}, &  \text{if } \beta\le\bar a(m),\\[2mm]
\beta\,\bar a(m), & \text{if }  \beta\ge\bar a(m).
\end{cases}
\end{equation}
The second equality follows from an elementary computation. Note
that $\Phi_h(\beta,\cdot)$ is continuous on $[-1,1]$, including
at $m=\pm1$ where $s(\pm1)=\bar a(\pm1)=0$ and therefore $\Phi_h(\beta,\pm1)=\pm\beta h$, by definition.

\begin{figure}[h]
\centering
\begin{tikzpicture}[>=stealth,scale=1]

  \def\sx{4}
  \def\sy{3}

  \draw[->] (-4.8,0) -- (4.8,0) node[right] {$m$};
  \draw[->] (0,-0.3) -- (0,3.9) node[right] {$\bar a(m)$};

  \foreach \x in {-4,4}
    \draw[gray!25] (\x,-0.15) -- (\x,3.7);

  \foreach \x/\xlab in {-4/-1,4/1} {
    \draw (\x,0.06) -- (\x,-0.06);
    \node[below] at (\x,-0.06) {\small $\xlab$};
  }
  \node[below] at (0,-0.06) {\small $0$};

  \draw[dashed,gray] (0,{\sy*sqrt(2*ln(2))}) -- (-4.8,{\sy*sqrt(2*ln(2))});
  \node[left] at (-4.85,{\sy*sqrt(2*ln(2))}) {\small $\sqrt{2\log 2}$};

  \draw[red,thick,smooth,samples=100,domain=-0.999:0.999,variable=\t]
    plot ({\sx*\t},{\sy*sqrt(2*( -((1+\t)/2)*ln((1+\t)/2) - ((1-\t)/2)*ln((1-\t)/2) ))});

  \def\mstar{0.717}
  \def\abarstar{0.9025}
  \coordinate (Mstar) at ({\sx*\mstar},0);
  \coordinate (PtStar) at ({\sx*\mstar},{\sy*\abarstar});

  \draw[dashed,red!70!black] (Mstar) -- (PtStar);
  \draw[dashed,red!70!black] (0,{\sy*\abarstar}) -- (PtStar);

  \filldraw[red!70!black] (PtStar) circle (1.3pt);
  \node[left,red!70!black] at (0,{\sy*\abarstar}) {\small $\bar a(m^*(h))=\beta_c(h)$};

  \filldraw[red!70!black] (Mstar) circle (1.3pt);
  \node[below,red!70!black,yshift=-2mm] at (Mstar) {\small $m^*(h)$};

\end{tikzpicture}
\caption{The function $\bar a(m) = \sqrt{2\,s(m)}$, for $m \in [-1,1]$.}
\label{fig:abar}
\end{figure}
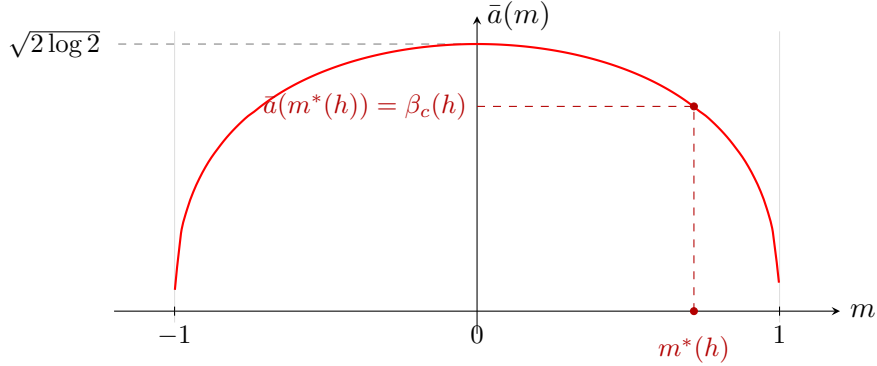

\medskip

\begin{prop}[Uniform upper bound on the restricted free energy]
\label{prop:upper}
Let $\beta>0$, $h>0$. There exists a constant $\kappa=\kappa(\beta)<\infty$
such that for every $\delta\in(0,1)$ there exists $n_0=n_0(\beta,\delta)\ge1$ such that
\[
\P\left(\exists\, m\in M_n:\ \frac1n\log Z_{\beta,h,n}(m)
>\Phi_h(\beta,m)+\kappa\delta\right)
\le \frac{\kappa n}{\delta}\,\ee^{-n\delta^2/2}, \qquad \forall \, n \ge n_0.
\]
In particular the bound
$\frac1n\log Z_{\beta,h,n}(m)\le\Phi_h(\beta,m)+\kappa\delta$ holds
simultaneously for all $m\in M_n$ with probability at least
$1-\kappa n\delta^{-1}\ee^{-n\delta^2/2}$.
\end{prop}

\medskip

\begin{proof}
Fix $\delta\in(0,1)$ and define $A:=\sqrt{2\log 2}$, so that
$\bar a(m)\le A$ for all $m$. Let $J:=\lceil (A+\delta)/\delta\rceil$ and
$a_j:=j\delta$ for all $0\le j\le J$, such that $a_J\ge A+\delta\ge\bar a(m)+\delta$
for every $m$. We also fix $\kappa=\kappa(\beta)$ large enough (but independent of  $\delta\in(0,1)$) such that
\begin{equation}
\label{eq:kappa}
J+1\le \frac{\kappa}{\delta}, \qquad (n+1)(J+2)\le \frac{\kappa n}{\delta}, \qquad 2(\beta+1) \le \kappa, \qquad \forall \delta\in(0,1), \, \forall \, n \ge 1.
\end{equation}

\medskip

{\it \underline{Step 1} : The good event.} Define
\[
\mathcal E_n:=\bigcap_{m\in M_n}\left(
\left\{N_n(\bar a(m)+\delta,m)=0\right\}\cap
\bigcap_{j=0}^{J}\left\{N_n(a_j,m)\le \ee^{n(s(m)-a_j^2/2+\delta)}\right\}\right), \qquad \forall \, n \ge 1.
\]
By Lemma  \ref{lem:firstmom} and Markov's inequality, for each $m \in M_n$ and each
$0 \leq j \leq J$,
\[
\P\left(N_n(a_j,m)\geq \ee^{n(s(m)-a_j^2/2+\delta)}\right)\le \ee^{-n\delta}, \qquad \forall \, n \ge 1,
\]
while, again by Lemma  \ref{lem:firstmom} with $a=\bar a(m)+\delta$,
\[
\P\left(N_n(\bar a(m)+\delta,m)\ge1\right)
\le \ee^{n\left(s(m)-\frac{(\bar a(m)+\delta)^2}{2}\right)}
=\ee^{-n\left(\delta\bar a(m)+\frac{\delta^2}{2}\right)}\le \ee^{-n\frac{\delta^2}{2}}, \qquad \forall \, n \ge 1,
\]
{\it uniformly in $m$}. This last bound is the reason for writing
$\delta^2/2$ in the statement, and it remains valid on the degenerate layers
$m=\pm1$ where $\bar a(m)=0$. A union bound over the
$(n+1)(J+2)\le \kappa n/\delta$ (by Equation \eqref{eq:kappa}) events gives
\begin{equation}
\label{eq:good_event}
\P(\mathcal E_n^c)\le \frac{\kappa n}{\delta} \, \ee^{-n\delta^2/2}, \qquad \forall \, n \ge 1.
\end{equation}

\medskip

{\it \underline{Step 2} : On  $\mathcal E_n$, slicing the layer.} Fix $n \geq 1$, $m\in M_n$ and let us 
decompose according to the values of $X_n(\sigma)$ :
\[
\ee^{-\beta n h m}Z_{\beta,h,n}(m)=\sum_{\sigma\in\Sigma_n(m)}\ee^{\beta X_n(\sigma)}
=\underbrace{\sum_{\sigma:\,X_n(\sigma)<0}}_{=:S_-}
+\sum_{j=0}^{J-1}\ \underbrace{\sum_{\sigma:\,a_jn\le X_n(\sigma)<a_{j+1}n}}_{=:S_j}
+\underbrace{\sum_{\sigma:\,X_n(\sigma)\ge a_Jn}}_{=:S_+},
\]
all sums being over $\sigma\in\Sigma_n(m)$. Make now the following observations.

\begin{itemize}
  \item[(a)] On the good event $\mathcal E_n$, we have 
  \[
S_+=0,
\]
since
$a_J\ge\bar a(m)+\delta$ and $N_n(\bar a(m)+\delta,m)=0$.

  \item[(b)] We have $S_-\le\#\Sigma_n(m)\le \ee^{ns(m)}$, by the entropy bound on multinomial
coefficients (see \cite[Theorem 11.1.3]{CoverThomas06}). Moreover
$\Phi_h(\beta,m)-\beta hm\ge s(m)$ is satisfied in both regimes of Equation \eqref{eq:Phi}
(indeed $s(m)+\beta^2/2\ge s(m)$ when $\beta\le\bar a(m)$ and $\beta\bar a(m)\ge\frac{\bar a(m)^2}{2}=s(m)$
when $\beta\ge\bar a(m)$). Hence, we get
  \[
S_-\le \ee^{n(\Phi_h(\beta,m)-\beta hm)}.
\]

  \item[(c)]  
On $\mathcal E_n$, and for all $0\le j\le J-1$,
\[
S_j\le \ee^{\beta a_{j+1}n}\,N_n(a_j,m)
\le \exp\left\{n\left(s(m)-\frac{a_j^2}{2}+\beta a_j+(\beta+1)\delta\right)\right\}.
\]
Moreover only the indices $j$ with $a_j\le\bar a(m)+\delta$ can contribute a
nonzero term (again because $N_n(\bar a(m)+\delta,m)=0$ on $\mathcal E_n$),
and, for such $j$, we have
\[
s(m)-\frac{a_j^2}{2}+\beta a_j
\le \max_{0\le a\le\bar a(m)+\delta}\left\{s(m)-\frac{a^2}{2}+\beta a\right\}
\le \max_{0\le a\le\bar a(m)}\left\{s(m)-\frac{a^2}{2}+\beta a\right\}+\beta\delta,
\]
the last inequality because, for $a\in[\bar a(m),\bar a(m)+\delta]$, one has
$s(m)-\frac{a^2}{2}+\beta a\le s(m)-\frac{\bar a(m)^2}{2}+\beta\bar a(m)+\beta\delta$ (by concavity of the parabola). Hence, by Equation \eqref{eq:Phi},
\[
S_j\le \exp\left\{n\left(\Phi_h(\beta,m)-\beta hm+(2\beta+1)\delta\right)\right\}, \qquad \forall \, 0\le j\le J-1.
\]
\end{itemize}

\medskip

{\it \underline{Step 3}.} Summing Observations (a)--(c) over the at most $J+1\le\kappa/\delta$  (by Equation \eqref{eq:kappa}) slices, we get, on $\mathcal E_n$,
\[
Z_{\beta,h,n}(m)\le \frac{\kappa}{\delta}\,
\exp\left\{n\left(\Phi_h(\beta,m)+(2\beta+1)\delta\right)\right\}, \qquad \forall \, n \ge 1, \; \forall \, m\in M_n.
\]
Pick $ n_0=n_0(\beta,\delta)$ large enough such that 
$\frac1n\log(\kappa/\delta)\le\delta$ for $n\ge n_0$. Then, on the good event $\mathcal E_n$,
$Z_{\beta,h,n}(m)\le
\exp\left\{n\left(\Phi_h(\beta,m)+2(\beta+1)\delta\right)\right\},$
for all $n \ge n_0$, all $m\in M_n$, which implies, by Equation \eqref{eq:kappa}, that, on $\mathcal E_n$,
\[
\frac1n\log Z_{\beta,h,n}(m)\le\Phi_h(\beta,m)+\kappa\delta, \qquad \forall \, n \ge n_0, \; \forall \, m\in M_n.
\]
This last result combined with Equation \eqref{eq:good_event} concludes the proof.
\end{proof}

\subsection{The variational problem}
\label{subsec:var_prob}

\medskip
\begin{prop}
\label{prop:var}
Let $h>0$ and recall $m^*=m^*(h)=\tanh(\beta_c h)$.
\begin{itemize}
\item[(i)] One has $\bar a(m^*)=\sqrt{2s(m^*)}=\beta_c$ or equivalently
$s(m^*)=\beta_c^2/2$. Moreover
$\psi_h'(\beta_c)=\beta_c+hm^*=\gamma_{\max}(h)$.
\item[(ii)] For every $\beta>0$, the function
$\phi(m):=\beta\left(\bar a(m)+hm\right)=\beta\left(\sqrt{2s(m)}+hm\right)$, see Figure \ref{fig:phi}, is
strictly concave on $(-1,1)$ and attains its maximum over $[-1,1]$ at the
unique point $m^*$. Moreover
$\phi(m^*)=\beta\left(\beta_c+hm^*\right)=\beta\gamma_{\max}(h)$.
\item[(iii)] Let $\beta>\beta_c$. Then $m\mapsto\Phi_h(\beta,m)$ attains its
maximum over $[-1,1]$ at the unique point $m^*$, and
\[
\max_{m\in[-1,1]}\Phi_h(\beta,m)=\Phi_h(\beta,m^*)=f(\beta,h)=\beta\gamma_{\max}(h).
\]
Consequently, for every $\eta\in(0,1)$,
\begin{equation}
\label{eq:gap}
c_1=c_1(\beta,h,\eta):=\Phi_h(\beta,m^*)-\sup_{m\in[-1,1],\,|m-m^*|>\eta}\Phi_h(\beta,m)>0 .
\end{equation}
\end{itemize}
\end{prop}

\medskip

\begin{proof}
{\it \underline{Proof of (i)}.} Since $\psi_h'(t)=t+h\tanh(ht)$, Equation \eqref{eq:beta_c} reads
$\beta_c^2+\beta_c h\tanh(\beta_c h)=\frac{\beta_c^2}{2}+\log2+\log\cosh(\beta_c h)$,
i.e.
\begin{equation}
\label{eq:betac2}
\frac{\beta_c^2}{2}=\log2+\log\cosh(\beta_c h)-\beta_c h\tanh(\beta_c h).
\end{equation}
On the other hand, Equation \eqref{eq:stanh} with $x=\beta_c h$, for which
$\tanh x=m^*$, gives
$s(m^*)=\log2+\log\cosh(\beta_c h)-\beta_c h\tanh(\beta_c h)$. Comparing with
Equation \eqref{eq:betac2} yields $s(m^*)=\beta_c^2/2$, i.e.\ $\bar a(m^*)=\beta_c$.
The identity
$\psi_h'(\beta_c)=\beta_c+h\tanh(\beta_c h)=\gamma_{\max}(h)$ is immediate.

\medskip

\noindent  {\it \underline{Proof of (ii)}.} The function $s$ is strictly concave and nonnegative on $(-1,1)$ and
$x\mapsto\sqrt x$ is concave and strictly increasing on $[0,\infty)$, so
$\bar a=\sqrt{2s}$ is strictly concave on $(-1,1)$; adding the linear term
$hm$ and the factor $\beta>0$ preserves strict concavity. Therefore, the function $\phi$ is strictly concave, but also differentiable on $(-1,1)$ and a stationary point
of $\phi$ in $(-1,1)$ satisfies, by Equation \eqref{eq:s-prime},
\[
\frac{s'(m)}{\sqrt{2s(m)}}+h=0
\iff \arctanh(m)=h\sqrt{2s(m)} .
\]
By (i), $m^*$ is a solution :
$\arctanh(m^*)=\beta_c h=h\sqrt{2s(m^*)}$ and therefore a stationary point
of $\phi$  in $(-1,1)$.
By strict concavity, $m^*$ is the unique stationary
point and the unique global maximizer of $\phi$ on
$(-1,1)$. Since $\phi$ is then strictly decreasing on $[m^*,1)$ and strictly
increasing on $(-1,m^*]$, continuity gives
$\phi(\pm1)<\phi(m^*)$, so the maximum over $[-1,1]$ is attained only at
$m^*$. Finally $\phi(m^*)=\beta(\beta_c+hm^*)=\beta\gamma_{\max}(h)$ by (i) (see also Figure \ref{fig:phi}).

\medskip

\noindent  {\it \underline{Proof of (iii)}.}  Split $[-1,1]=A_\beta\cup F_\beta$ with
\[
A_\beta:=\{m:\bar a(m)\ge\beta\}=\{m:s(m)\ge\beta^2/2\},
\qquad F_\beta:=\{m:\bar a(m)<\beta\}.
\]
Note first that $m^*\in F_\beta$, because $\bar a(m^*)=\beta_c<\beta$
by (i). Since $\max_{m\in [-1,1]}s(m)=s(0)=\log 2$, one has $A_\beta=\emptyset$ if
and only if $\beta>\sqrt{2\log 2}$, and we treat the two cases separately.

\medskip
\textbf{Case 1 : if $\beta>\sqrt{2\log 2}$.} Then $\bar
a(m)=\sqrt{2s(m)}\le\sqrt{2\log 2}<\beta$ for every $m\in[-1,1]$, so
$A_\beta=\emptyset$ and $F_\beta=[-1,1]$. By Equation \eqref{eq:Phi}, $\Phi_h(\beta,\cdot)=\phi$
on all of $[-1,1]$, and (ii) gives directly that $\Phi_h(\beta,\cdot)$
attains its maximum $\beta\gamma_{\max}(h)$ over $[-1,1]$ at the unique point
$m^*$. No further argument is needed in this case.

\medskip

\textbf{Case 2 : if $\beta_c<\beta\le\sqrt{2\log 2}$} (a non-empty range,
since $h>0$ forces $m^*>0$, hence $s(m^*)<\log 2$ and $\beta_c<\sqrt{2\log
2}$ by (i)). Then $A_\beta\neq\emptyset$ and, since $s$ is even, strictly concave, maximal at $0$ and vanishing at $\pm1$,
$A_\beta=[-m_+(\beta),m_+(\beta)]$ is a compact interval, where $m_+(\beta)\in[0,1)$ is well defined by $s(m_+(\beta))=\beta^2/2$, existence and uniqueness
following from the fact that $s$ is continuous and strictly decreasing from
$\log 2$ to $0$ on $[0,1]$, together with $\beta^2/2\le\log 2$.

\medskip

\noindent  {\it \underline{On $F_\beta$} :} $\Phi_h(\beta,\cdot)=\phi$ there by Equation \eqref{eq:Phi}, so
by (ii) $\sup_{F_\beta}\Phi_h(\beta,\cdot)=\phi(m^*)=\beta\gamma_{\max}(h)$,
attained only at $m^*$.

\medskip

\noindent {\it \underline{On $A_\beta$} :} there
$\Phi_h(\beta,m)=\phi_{\mathrm{an}}(m):=s(m)+\frac{\beta^2}{2}+\beta hm$,
which is strictly concave on $(-1,1)$ with unique unconstrained maximizer
$\tanh(\beta h)$ (by Equation \eqref{eq:s-prime},  one has $s'(m)+\beta h=0\iff m=\tanh(\beta h)$).
We claim that 
\[
\tanh(\beta h)\notin A_\beta, \qquad \mathrm{and} \qquad  \tanh(\beta h)>m_+(\beta).
\]
 Indeed, by Equation \eqref{eq:stanh} with
$x=\beta h$,
\[
s\left(\tanh(\beta h)\right)-\frac{\beta^2}{2}
=\log2+\log\cosh(\beta h)-\beta h\tanh(\beta h)-\frac{\beta^2}{2}
=\psi_h(\beta)-\beta\psi_h'(\beta)<0,
\]
by Lemma \ref{lem:Psi-h} ($\beta>\beta_c$ here). Therefore
$s(\tanh(\beta h))<\beta^2/2=s(m_+(\beta))$, i.e.\ $\tanh(\beta h)\notin A_\beta$. 
As $s$ is strictly decreasing on
$[0,1]$ and both numbers are nonnegative, this gives
$\tanh(\beta h)>m_+(\beta)$. 

By strict
concavity, the maximum of $\phi_{\mathrm{an}}$ over the interval $A_\beta$ is
thus attained at the endpoint closest to $\tanh(\beta h)$, namely
$m_+(\beta)$ :
\[
\sup_{A_\beta}\Phi_h(\beta,\cdot)=\phi_{\mathrm{an}}\left(m_+(\beta)\right)
=\phi\left(m_+(\beta)\right),
\]
the last equality because $\bar a(m_+(\beta))=\beta$, so that the two
expressions in Equation \eqref{eq:Phi} agree there. Finally $\beta>\beta_c$ implies
$s(m_+(\beta))=\beta^2/2>\beta_c^2/2=s(m^*)$, hence $m_+(\beta)<m^*$ (again
both nonnegative and $s$ strictly decreasing on $[0,1]$), so
$m_+(\beta)\neq m^*$ and, by the uniqueness in (ii),
$\phi(m_+(\beta))<\phi(m^*)$.

Combining the two cases, $\Phi_h(\beta,\cdot)$ attains its maximum
$\beta\gamma_{\max}(h)$ only at $m^*$. That this equals $f(\beta,h)$ for
$\beta>\beta_c$ is the content of Theorem \ref{thm:theoremB} (and is here re-derived). Finally, for $\eta\in(0,1)$ the set $\{m\in[-1,1]:|m-m^*|\ge\eta\}$
is non-empty --- it contains $m=-1$, since $|-1-m^*|=1+m^*>1>\eta$ --- and
compact.
Equation \eqref{eq:gap} then follows from the continuity of
$\Phi_h(\beta,\cdot)$ on 
$\{m\in[-1,1]:|m-m^*|\ge\eta\}$ together with the uniqueness of the
maximizer.
\end{proof}

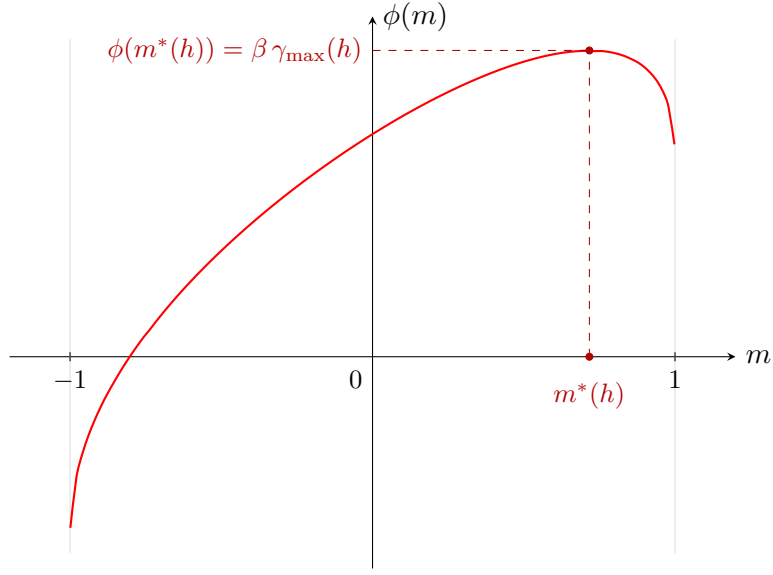
\begin{figure}[h]
\centering
\begin{tikzpicture}[>=stealth,scale=1]

  \def\sx{4}
  \def\sy{2.5}

  \draw[->] (-4.8,0) -- (4.8,0) node[right] {$m$};
  \draw[->] (0,-2.8) -- (0,4.5) node[right] {$\phi(m)$};

  \foreach \x in {-4,4}
    \draw[gray!25] (\x,-2.6) -- (\x,4.2);

  \foreach \x/\xlab in {-4/-1,4/1} {
    \draw (\x,0.06) -- (\x,-0.06);
    \node[below] at (\x,-0.06) {\small $\xlab$};
  }
  \node[below left] at (0,-0.06) {\small $0$};

  \draw[red,thick,smooth,samples=100,domain=-0.999:0.999,variable=\t]
    plot ({\sx*\t},{\sy*( sqrt(2*( -((1+\t)/2)*ln((1+\t)/2) - ((1-\t)/2)*ln((1-\t)/2) )) + \t )});

  \def\mstar{0.717}
  \def\phistar{1.6200}
  \coordinate (Mstar) at ({\sx*\mstar},0);
  \coordinate (PtMax) at ({\sx*\mstar},{\sy*\phistar});

  \draw[dashed,red!70!black] (Mstar) -- (PtMax);
  \draw[dashed,red!70!black] (0,{\sy*\phistar}) -- (PtMax);

  \filldraw[red!70!black] (PtMax) circle (1.3pt);
  \node[left,red!70!black] at (0,{\sy*\phistar}) {\small $\phi(m^*(h))=\beta\,\gamma_{\max}(h)$};

  \filldraw[red!70!black] (Mstar) circle (1.3pt);
  \node[below,red!70!black,yshift=-2mm] at (Mstar) {\small $m^*(h)$};

\end{tikzpicture}
\caption{The function $\phi(m) = \beta(\bar a(m) + hm)$, for $m \in [-1,1]$, illustrated for $\beta=h=1$. Unlike $\bar a(m)$, the curve is not symmetric : the linear term $hm$ tilts the maximum towards positive $m$.}
\label{fig:phi}
\end{figure}

\subsection{Lower bound on the total partition function}

\medskip

\begin{prop}
\label{prop:lower}
Let $\beta>0$, $h>0$ and $\delta>0$. There exists
$n_1=n_1(\beta,h,\delta)<\infty$ such that 
\[
\P\left(\frac1n\log Z_{\beta,h,n}\le f(\beta,h)-\delta\right)
\le \ee^{-\frac{n\delta^2}{8\beta^2}}, \qquad \forall \, n\ge n_1.
\]
\end{prop}

\medskip

\begin{proof}
This is standard computation. 
View $F_n:=\frac1n\log Z_{\beta,h,n}$ as a function of the Gaussian vector
$\mathbf U=(U(v))_{1\le|v|\le n}$. For a configuration $v$ with $|v|=k\le n$,
\[
\frac{\partial}{\partial U(v)}\log Z_{\beta,h,n}
=\beta\, \mathcal  G_{\beta,h,n}\left(\mathcal B(v)\right),
\qquad \mathcal B(v):=\{\sigma\in\Sigma_n:\sigma\succ v\},
\]
since $U(v)$ appears in $H_n(\sigma,h)$ exactly for the descendants of $v$.
For each fixed $k$, the sets $(\mathcal B(v))_{|v|=k}$ form a partition of
$\Sigma_n$, so that 
$\sum_{|v|=k}\mathcal  G_{\beta,h,n}(\mathcal B(v))^2\le
\sum_{|v|=k}\mathcal  G_{\beta,h,n}(\mathcal B(v))=1$. Hence
$\|\nabla\log Z_{\beta,h,n}\|_2^2\le\beta^2 n$  and $F_n$ is
$\frac{\beta}{\sqrt n}$-Lipschitz. The Gaussian concentration inequality, see  \cite[Theorem 2.2.4]{talagrand2003},
gives
\[
\P\left(F_n\le\E [ F_n]-\frac{\delta}{2}\right)
\le \exp\left\{-\frac{n\delta^2}{8\beta^2}\right\}.
\]
By Theorem \ref{thm:theoremB} (convergence in $L^1$), $\E[ F_n]\to f(\beta,h)$, when $n \to \infty$, such that
$\E[ F_n]\ge f(\beta,h)-\frac\delta2$, for $n\ge n_1=n_1(\beta,h,\delta)$, and the conclusion follows.
\end{proof}

\subsection{Proof of Theorem \ref{thm:theorem1}}

Fix $h>0$, $\beta>\beta_c(h)>0$ and $\eta>0$. Since the event $\{|y_n(\sigma)-m^*|>\eta\}$ is
non-increasing in $\eta$, an exponential bound established for some
$\eta'\in(0,\eta]$ implies the same bound for $\eta$. Therefore, we
assume $\eta\in(0,1)$, which is what Proposition \ref{prop:var}(iii) requires. Let $c_1=c_1(\beta,h,\eta)>0$ be the gap in Equation
\eqref{eq:gap} of Proposition \ref{prop:var} (iii), let $\kappa=\kappa(\beta)$
be the constant of Proposition \ref{prop:upper}, and set
\begin{equation}
\label{eq:delta}
\delta:=\min\left\{\frac{c_1}{4\kappa},\ \frac{c_1}{4},\ \frac12\right\} \in (0,1).
\end{equation}
Introduce the two events
\[
\mathcal H_n^{\mathrm{up}}
:=\left\{\frac{1}{n}\log Z_{\beta,h,n}(m)
\le\Phi_h(\beta,m)+\kappa\delta, \  \forall \, m\in M_n \right\},
\quad
\mathcal H_n^{\mathrm{low}}
:=\left\{\frac{1}{n}\log Z_{\beta,h,n}\ge f(\beta,h)-\delta\right\},
\]
and define
\[
\mathcal H_n:=\mathcal H_n^{\mathrm{up}}\cap\mathcal H_n^{\mathrm{low}}, \qquad \forall \, n\ge 1.
\]

\noindent Then, consider $n_0(\beta,\delta)$ introduced in Proposition  \ref{prop:upper}, $n_1(\beta,h,\delta)$  introduced in Proposition  \ref{prop:lower} and define
$
n_2(\beta,h,\delta) := \max\{n_0(\beta,\delta),n_1(\beta,h,\delta)\}.
$

\medskip

\noindent {\it \underline{ $\mathcal H_n$ is a good event}.}
By Proposition \ref{prop:upper} and Proposition \ref{prop:lower},
\begin{equation}
\label{eq:probHn}
\P\left(\mathcal H_n^c\right)
\le \frac{\kappa n}{\delta} \, \ee^{-n\delta^2/2}
+\ee^{-n\delta^2/(8\beta^2)},
\qquad \forall \, n\ge n_2(\beta,h,\delta).
\end{equation}

\medskip

\noindent {\it \underline{On $\mathcal H_n$}.} Fix $n \geq 1$ and let $m\in M_n$ with $|m-m^*|>\eta$. By
Equation \eqref{eq:gap}  and Proposition \ref{prop:var} (iii), we have
$\Phi_h(\beta,m)\le\Phi_h(\beta,m^*)-c_1=f(\beta,h)-c_1$, hence
\[
Z_{\beta,h,n}(m)\le\exp\left\{n(f(\beta,h)-c_1+\kappa\delta)\right\}.
\]
Since moreover $Z_{\beta,h,n}\ge\exp\left\{n(f(\beta,h)-\delta)\right\}$ on
$\mathcal H_n^{\mathrm{low}}$, and since there are at most $\#M_n=n+1$ layers, we obtain
\[
\mathcal G_{\beta,h,n}\left\{|y_n(\sigma)-m^*|>\eta\right\}
=\frac{\displaystyle\sum_{m\in M_n:\,|m-m^*|>\eta}Z_{\beta,h,n}(m)}{Z_{\beta,h,n}}
\le (n+1)\,\ee^{-n(c_1-\kappa\delta-\delta)}
\le (n+1)\,\ee^{-nc_1/2},
\]
where the last inequality uses $\kappa\delta\le c_1/4$ and
$\delta\le c_1/4$ (see Equation \eqref{eq:delta}).

\medskip

\noindent {\it \underline{Conclusion}.} Since $\mathcal  G_{\beta,h,n}(\, \cdot \,)$ is a probability measure, the
quantity under consideration is bounded by $1$, so splitting on
$\mathcal H_n$ and using Equation \eqref{eq:probHn} yields
\[
\E  \left[\mathcal G_{\beta,h,n}\left\{|y_n(\sigma)-m^*|>\eta\right\}\right]
\le (n+1)\ee^{-nc_1/2}
+\frac{\kappa n}{\delta} \, \ee^{-n\delta^2/2}
+\ee^{-n\delta^2/(8\beta^2)},
\qquad \forall \, n\ge n_2(\beta,h,\delta).
\]
All three terms decay exponentially fast in $n$, which yields Equation \eqref{eq:quant} with
any
\[
c=c(\beta,h,\eta)< \min\left\{\frac{c_1}{2},\ \frac{\delta^2}{2},\
\frac{\delta^2}{8\beta^2}\right\}
\]
and a suitable $C=C(\beta,h,\eta)$.

For the {\it almost-sure} convergence, it is enough to observe that,  for all $\varepsilon >0$, Markov's inequality implies
\[
\sum_{n \ge 1} \p\left(\left\vert \mathcal G_{\beta,h,n}\left\{|y_n(\sigma)-m^*|>\eta\right\}\right\vert > \varepsilon \right) \le \frac{1}{\varepsilon} \sum_{n \ge 1} \E  \left[\mathcal G_{\beta,h,n}\left\{|y_n(\sigma)-m^*|>\eta\right\}\right] < \infty,
\]
by Equation \eqref{eq:quant} and then apply the Borel-Cantelli lemma.  \hfill$\square$

\subsection{An immediate corollary}
\label{subsec:coro_immediate}

This one-replica concentration result has an immediate two-replica consequence, which
we record here since it is the natural (but, as Remark \ref{rem:counter-ex} will show, insufficient)
starting point for the proof of Theorem \ref{thm:theorem2}.

\medskip

\begin{cor}[Joint concentration of the two marginal magnetizations]
\label{cor:marginals}
Fix $h>0$ and $\beta>\beta_c(h)>0$. Then, for every $\eta>0$,
\[
\lim_{n\to\infty}\mathbb E\left[\mathcal G_{\beta,h,n}^{\otimes 2}\left\{|y_n(\sigma)-m^*|>\eta
\; \;  {\rm or }\; \;  |y_n(\sigma')-m^*|>\eta\right\}\right]=0.
\]
\end{cor}

\medskip

\begin{proof}
Since $\mathcal  G_{\beta,h,n}^{\otimes 2} = \mathcal G_{\beta,h,n}\otimes \mathcal G_{\beta,h,n}$, a union bound
gives
\[
\mathcal  G_{\beta,h,n}^{\otimes 2}\left\{|y_n(\sigma)-m^*|>\eta \; \; {\rm or }\; \;  |y_n(\sigma')-m^*|
>\eta\right\}\le 2\,\mathcal  G_{\beta,h,n}\{|y_n(\sigma)-m^*|>\eta\},
\]
and the conclusion follows from Theorem \ref{thm:theorem1}.
\end{proof}

\noindent Corollary \ref{cor:marginals} gives joint concentration of the two {\it marginal}
magnetizations near $(m^*,m^*)$. As Remark \ref{rem:counter-ex} shows by an elementary example, this
does {\it not} imply that the overlap $r_n(\sigma,\sigma')$ concentrates near
$(m^*)^2$ : concentration of two marginals says nothing about their joint (site-by-site)
correlation. Ruling out this naive route is precisely the starting point of Section \ref{sec:overlap}.

\section{Proof of Theorem \ref{thm:theorem2}}
\label{sec:overlap}

\subsection{Two-replica estimates}
\label{subsec:2replica}

\begin{rem}[Why a one-replica statement does not suffice]
\label{rem:counter-ex}
It is tempting to argue that on $\{q_n\le\varepsilon\}$ the two replicas branch at an early generation $k$,
so that (conditionally on the disorder) their tails are independent. Hence, since each tail has empirical
magnetization close to $m^*$ by Theorem \ref{thm:theorem1} --- or
more precisely, jointly, by Corollary \ref{cor:marginals} --- the overlap $r_n$ should be close to $(m^*)^2$. This is
not valid, indeed if $(\tau,\tau')\sim\mu_1\otimes\mu_2$ on $\{-1,1\}^N$, then
\[
\mathbb E\left[\frac1N\sum_i \tau_i\tau_i'\right]=\frac1N\sum_i\langle\tau_i\rangle_{\mu_1}
\langle\tau_i\rangle_{\mu_2},
\]
which is a site-by-site correlation, not the product of the two {\it global} magnetizations. Concentration
of $\frac1N\sum_i\tau_i$ under each $\mu_j$ says nothing about it. (Take $\mu_1=\mu_2$ uniform on the
two alternating sequences $\pm(1,-1,1,\dots) \in \{-1,1\}^{N}$ with $N \ge 2$ even : both magnetizations $y_N(\tau)$ and  $y_N(\tau')$ vanish identically ($N$ is even), $E\left[ \tau_i\tau_i'\right] = 0$ so that $E\left[r_N \right] = 0$ but
$r_N=\frac1N\sum_i\tau_i\tau_i'=\pm1$ with probability $1/2$, i.e. its law $\frac{1}{2}(\delta_{-1}+\delta_{1})$ is far from concentrated at $0$.) What forces $r_n\to(m^*)^2$ is the entropic cost of a joint
empirical {\it type} that is {\it not} a product --- the content of the subadditivity inequality in Equation
\eqref{eq:subadd} below, quantified by the variational problem of Proposition \ref{prop:variational}.
\end{rem}

\subsubsection*{Empirical types of a pair}

Let $n\ge2$, $(\sigma,\sigma')\in\Sigma_n^2$ so that 
\[
k:=|\sigma\wedge\sigma'|\in \{0,\dots,n-1\} \qquad {\rm and } \qquad  N=N(n,k):=n-k-1\ge0.
\]
The {\it type} of the pair $(\sigma,\sigma')$ is 
$$
{\bf q}(\sigma,\sigma')=\left(q_{ab}(\sigma,\sigma')\right)_{a,b\in\{-1,1\}},
$$
 defined, for $N\ge1$, by
\[
q_{ab}(\sigma,\sigma'):=\frac1N\,\#\{k+1<i\le n:\ \sigma_i=a,\ \sigma_i'=b\}, \qquad \forall \, a,b  \in\{-1,1\}.
\]
Let $\mathcal P$ be the simplex of probability measures on $\{-1,1\}^2$ and 
\[
\mathcal P_N:=\{{\bf q} \in\mathcal P: Nq_{ab}\in\N, \ \forall \, a,b \in\{-1,1\}\},
\]
so  that
\begin{equation}
\label{eq:PN-count}
\#\mathcal P_N\le (N+1)^3\le(n+1)^3. 
\end{equation}
For ${\bf q} \in\mathcal P$, set
\[
m_1({\bf q}):=\sum_{a,b \in\{-1,1\}}a\,q_{ab},\qquad m_2({\bf q}):=\sum_{a,b \in\{-1,1\}}b\,q_{ab},
\]
and
\[
\rho({\bf q}):=\sum_{a,b \in\{-1,1\}}ab\,q_{ab},\qquad
S({\bf q}):=-\sum_{a,b \in\{-1,1\}}q_{ab}\log q_{ab}.
\]
Recalling the classical information-theoretic
 facts from Appendix \ref{app:subadditivity}, the marginals of ${\bf q}$ have entropies $s(m_1({\bf q}))$, $s(m_2({\bf q}))$ (the binary entropy function $s$ being defined in Equation \eqref{eq:binary-entropy}), subadditivity of entropy (see Lemma \ref{lem:subadditivity}) gives
\begin{equation}
S({\bf q})\le s(m_1({\bf q}))+s(m_2({\bf q})),
 \label{eq:subadd}
\end{equation}
with equality  if and only if
\begin{equation*}
q_{ab}=\left(\frac{1+am_1({\bf q})}{2}\right)\left(\frac{1+bm_2({\bf q})}{2}\right), \qquad \forall \, a,b  \in\{-1,1\},
\end{equation*}
i.e.  if and only if ${\bf q}$ is a {\it product type}, in which case $\rho({\bf q})=m_1({\bf q})m_2({\bf q})$. Since, by definition of $k=|\sigma\wedge\sigma'|$, 
$\sigma_i\sigma_i'=1$ for $i\le k$ and $\sigma_{k+1}\sigma_{k+1}'=-1$,
\begin{equation} \label{eq:overlap-type}
r_n(\sigma,\sigma')=\frac{k-1}n+\frac Nn\,\rho\left({\bf q}(\sigma,\sigma')\right),
\qquad\text{hence}\qquad
\left|r_n(\sigma,\sigma')-\rho\left({\bf q}(\sigma,\sigma')\right)\right|\le\frac{2(k+1)}n. 
\end{equation}
Splitting the Hamiltonian at generation $k+1$, write $H_n(\sigma,h)=H_{k+1}(\sigma|_{k+1},h)
+\widetilde H_k(\sigma)$ with $\widetilde H_k(\sigma):=\sum_{k+1<i\le n}(U(\sigma|_i)+h\sigma_i)$,
$\widetilde X_k(\sigma):=\sum_{k+1<i\le n}U(\sigma|_i)$ and set, for ${\bf q}\in\mathcal P_N$,
\begin{equation}
 \label{eq:Wkn-def}
W_{k,n}({\bf q}):=\sum_{\substack{(\sigma,\sigma')\in\Sigma_n^2:\ |\sigma\wedge\sigma'|=k\\
{\bf q}(\sigma,\sigma')={\bf q}}} \ee^{\beta\widetilde H_k(\sigma)+\beta\widetilde H_k(\sigma')}.
\end{equation}
Because $\widetilde H_k(\sigma)$ and $\widetilde H_k(\sigma')$ involve two disjoint families of edges,
for each fixed pair $\widetilde X_k(\sigma)$ and $\widetilde X_k(\sigma')$ are independent $\mathcal N(0,N)$
variables. This is the only probabilistic input below.

\subsubsection*{The two-replica first moment and the head}

For $0\le k<n$, $N=N(n,k):=n-k-1\ge1$, ${\bf q} \in\mathcal P_N$, $a_1,a_2\ge0$, let
\[
\widetilde N_{k,n}({\bf q};a_1,a_2):= \# \left\{(\sigma,\sigma')\in\Sigma_n^2 \, : \, |\sigma\wedge\sigma'|=k  \, , \,  {\bf q}(\sigma,\sigma')={\bf q}  \, , \, \widetilde X_k(\sigma)\ge a_1N   \, , \,  \widetilde X_k(\sigma')\ge a_2N\right\}.
\]

\medskip

\begin{lem}[Two-replica first moment]
\label{lem:firstmoment}
For every $n\ge2$, every $0\le k<n$ with $N\ge1$, every ${\bf q} \in\mathcal P_N$ and all $a_1,a_2\ge0$,
\begin{equation}
\label{eq:first-moment-bound}
\mathbb E\left[\widetilde N_{k,n}({\bf q};a_1,a_2)\right]\le
\exp\left\{(k+1)\log2+N\left(S({\bf q})-\frac{a_1^2+a_2^2}2\right)\right\}. 
\end{equation}
\end{lem}

\medskip

\begin{proof}
Such a pair is determined by its common ancestor ($2^k$ choices), the label $\sigma_{k+1}$ ($2$
choices, forcing $\sigma_{k+1}'=-\sigma_{k+1}$), and for $k+1<i\le n$ the value $(\sigma_i,\sigma_i')$,
constrained to equal each $(a,b)$ exactly $Nq_{ab}$ times, for $a,b \in \{-1,1\}$. Hence we have at most
$2^{k+1}\ee^{NS({\bf q})}$ such pairs --- again by the entropy bound on multinomial
coefficients (see \cite[Theorem 11.1.3]{CoverThomas06}). For each pair, $\widetilde X_k(\sigma),\widetilde X_k(\sigma')$ are independent $\mathcal N(0,N)$, so
$\P(\widetilde X_k(\sigma)\ge a_1N \, ; \, \widetilde X_k(\sigma')\ge a_2N)\le \ee^{-Na_1^2/2}\ee^{-Na_2^2/2}$
by the Gaussian tail bound. Multiplying and summing gives Equation  \eqref{eq:first-moment-bound}.
\end{proof}

\medskip

\begin{lem}[Uniform control of the head]
\label{lem:head}
Fix $h>0$. Set $\Theta_0=\Theta_0(h):=\sqrt{2\log2}+h$ and
\[
\widetilde{\mathcal K}_n=\widetilde{\mathcal K}_n(\delta):=\{H_j(\sigma,h)\le\Theta_0 j+\delta n, \ \forall\, 0\le j\le n,\ \forall\, \sigma \in\Sigma_j\},  \qquad \forall \, n\ge1 \, , \, \forall \, \delta>0.
\]
 Then
\[
\P(\widetilde{\mathcal K}_n(\delta)^c)\le n \, \ee^{-\sqrt{2\log2}\,\delta n}, \qquad \forall \, n\ge1.
\]
\end{lem}

\medskip

\begin{proof}
Fix $1\le j\le n$. Since $H_j(\sigma,h)\le X_j(\sigma)+hj$ and $X_j(\sigma)\sim \mathcal N(0,j)$, a union bound over $2^j$
vertices and the classical Gaussian tail bound give, with $t=\sqrt{2\log2}\,j+\delta n$,
\[
\P\left(\exists \, \sigma \in\Sigma_j: X_j(\sigma)>t\right)\le 2^j \, \ee^{-\frac{t^2}{2j}}
=2^j\exp\left\{-\log2\, j-\sqrt{2\log2}\,\delta n-\frac{(\delta n)^2}{2j}\right\}
\le \ee^{-\sqrt{2\log2}\,\delta n}.
\]
Summing over $1\le j\le n$ concludes the proof (the inequality is deterministic and trivial for $j=0$).
\end{proof}

\subsubsection*{The rate functions $g_\beta$, $g_\beta^{(2)}$}

For $S\ge0$, let
\begin{equation}
\label{eq:gbeta-def}
g_\beta(S):=\max_{0\le a\le\sqrt{2S}}\left\{S-\frac{a^2}2+\beta a\right\}
=\begin{cases} S+\beta^2/2, & S\ge\beta^2/2,\\ \beta\sqrt{2S}, & S\le\beta^2/2,\end{cases}
\end{equation}
and
\begin{equation}
\label{eq:g2beta-def}
g_\beta^{(2)}(S):=\max_{\substack{a_1,a_2\ge0\\ a_1^2+a_2^2\le2S}}
\left\{S-\frac{a_1^2+a_2^2}2+\beta(a_1+a_2)\right\}. 
\end{equation}
From Subsection \ref{subsec:2D-encoding}, recall that $\Phi_h(\beta,m)=\beta hm+g_\beta(s(m))$, for all $m\in[-1,1]$.
The two branches of Equation \eqref{eq:gbeta-def} agree, together with their derivatives (both equal to $\beta^2$,
resp. to $1$), at $S=\beta^2/2$, so $g_\beta$ is $C^1$; moreover each branch is concave (an affine map, resp.\
$S\mapsto\beta\sqrt{2S}$, concave as $\sqrt\cdot$ composed with a linear map), so
\begin{equation}
\label{eq:gbeta-props}
g_\beta\ \text{is continuous, strictly increasing, concave on }[0,\infty),\ \ g_\beta(0)=0,\ \
g_\beta(S)\le S+\beta\sqrt{2S}. 
\end{equation}
(The upper bound is a consequence of $\beta^2/2 \le \beta^2 \le \beta \sqrt{2S}$ for $S\ge\beta^2/2$, and trivial otherwise.) Being concave with
$g_\beta(0)=0$, $g_\beta$ is subadditive (simple exercise). Writing $(a_1,a_2)=\varrho(\cos\theta,\sin\theta)$, at fixed
$0 \le \varrho\le\sqrt{2S}$ the quantity $a_1+a_2=\varrho(\cos\theta+\sin\theta)$ is maximal, equal to $\sqrt2\varrho$
at $\theta=\pi/4$. Hence, substituting $\varrho=\sqrt2 a$,
\[
g_\beta^{(2)}(S)=\max_{0\le\varrho\le\sqrt{2S}}\left\{S-\frac{\varrho^2}2+\sqrt2\beta\varrho\right\}
=\max_{0\le a\le\sqrt S}\{S-a^2+2\beta a\}=2\max_{0\le a\le\sqrt S}\left\{\frac S2-\frac{a^2}2+\beta a\right\},
\]
so that we get the key identity :
\begin{equation}
\label{eq:g2-identity}
g_\beta^{(2)}(S)=2\,g_\beta(S/2), \qquad \forall \, S\ge0. 
\end{equation}
Consequently, by Equation \eqref{eq:g2-identity}, subadditivity of $g_\beta$ and Equation \eqref{eq:gbeta-props},
\begin{equation}
 \label{eq:g2-subadd}
g_\beta^{(2)}(S+\delta)-g_\beta^{(2)}(S)=2\left[g_\beta\left(\frac{S+\delta}{2}\right)-g_\beta\left(\frac
{S}{2}\right)\right]\le2\,g_\beta(\delta/2)\le\delta+2\beta\sqrt\delta, \qquad \forall \, S,\delta\ge0.
\end{equation}
Finally, for $\beta,h>0$ and ${\bf q}\in\mathcal P$, consider the rate function
\begin{equation}
\label{eq:Xi-def}
\Xi_{\beta,h}({\bf q}):=\beta h\left(m_1({\bf q})+m_2({\bf q})\right)+g_\beta^{(2)}\left(S({\bf q})\right). 
\end{equation}

\subsubsection*{The uniform upper bound}

\medskip

\begin{prop}[Uniform upper bound on the restricted two-replica partition functions]
\label{prop:unifbound}
Let $\beta,h>0$.
There is $\widetilde \kappa=\widetilde \kappa(\beta)<\infty$ such that for every $\delta\in(0,1)$ and every $\varepsilon\in\left(0,
\min\{\frac{1}{8},\frac{\delta}{8\log2}\}\right)$ there is $n_3=n_3(\beta,\delta)\ge1$ such that, if we define 
\[
\mathcal M_n:=\left\{ 
\frac1{n-k-1}\log W_{k,n}({\bf q})\le \Xi_{\beta,h}({\bf q})+\widetilde \kappa\sqrt\delta, \quad \forall \,0\le k\le\varepsilon n,\ \forall \, {\bf q}\in\mathcal P_{n-k-1} \right\},     
\]
we have
\[\P(\mathcal M_n^c)\le\frac{\widetilde \kappa n^4}{\delta^2}
\ee^{-n\delta/2}, \qquad \forall \, n \ge n_3.
\]

\end{prop}

\medskip

\begin{proof}
Here $\widetilde C=\widetilde C(\beta)$ denotes a finite constant, possibly changing from line to line. Fix $\delta,\varepsilon$ as
above and take $n$ large enough such that $\varepsilon n\ge1$ and
\begin{equation}
\label{eq:N-lower}
N=n-k-1\ge(1-\varepsilon)n-1\ge\frac{3}{4} n, \qquad \forall \, 0\le k\le\varepsilon n. 
\end{equation}
Let $A := 2\sqrt{\log 2}$, so that $S({\bf q})\le\log 4 = A^2/2$ for all ${\bf q}\in P$ (see \cite[Theorem 2.6.4]{CoverThomas06} for the inequality), and let $J:=\lceil(A+1+\delta)/\delta\rceil$,
$a_j:=j\delta$ for $0\le j\le J$ (so $a_J\ge A+1+\delta$).

\bigskip

\noindent  {\it \underline{On the treatment of negative energies}.} Before starting the proof, let us isolate and
explain in detail a single mechanism, used twice below, which handles the pairs $(\sigma,\sigma')$
for which one or both of the recentred energies $ \widetilde X_k(\sigma)$, $ \widetilde X_k(\sigma')$ are
negative --- the two-replica counterpart of the term $S^-$ in the proof of
Proposition \ref{prop:upper}.
In the one-replica proof, the negative-energy term $S^-=\sum_{\sigma:X_n(\sigma)<0} \ee^{\beta
X_n(\sigma)}$ was controlled  {\it without any probabilistic input} : since $\beta>0$ and
$X_n(\sigma)<0$, every summand satisfies $ \ee^{\beta X_n(\sigma)}\le 1$, so that $S^-$ is bounded
by the cardinality of the layer, $S^-\le\#\Sigma_n(m)\le  \ee^{ns(m)}$ --- a pure counting
(entropy) bound, coming for free from Step 2(b) of that proof. We want an analogous statement
here : pairs $(\sigma,\sigma')$ of a prescribed type ${\bf q}$ with, say, $ \widetilde X_k(\sigma)<0$, should be
controlled by the count of such pairs (irrespective of the value of $ \widetilde X_k(\sigma')$),
weighted by the trivial bound $ \ee^{\beta  \widetilde X_k(\sigma)}\le 1$ on the negative coordinate.

A naive way to obtain this would be to split the argument into the four cases
(positive/positive, positive/negative, negative/positive, negative/negative) and, in the three
cases involving a negative coordinate, redo by hand a one-sided version of the computation below.
This would quadruple the length of the proof for no mathematical gain. Instead we absorb all four
cases into a single, uniform argument by adjoining to the finite grid $\{a_0,a_1,\dots,a_J\}$,
with $a_j=j\delta$, one extra  {\it formal} symbol, denoted $*$, which will always carry the
numerical value $a_*:=0$, but will index a  {\it different} subset of $\mathbb R$ than $a_0$ does.

\medskip

\noindent  {\it \underline{Extended grid}.} Set $\mathcal J:=\{*,0,1,\dots,J\}$ and $a_*:=0$. Partition
$\mathbb R$ into the $J+2$ intervals
\[
I_* := (-\infty,0), \qquad I_j := [a_j,a_{j+1}) \ \ (0\le j\le J-1), \qquad I_J := [a_J,\infty),
\]
indexed by $\mathcal J$ so that every
pair $(\sigma,\sigma')$ has $\left( \widetilde X_k(\sigma)/N, \widetilde X_k(\sigma')/N\right)$ lying in
exactly one {\it cell} of the form $I_{j_1}\times I_{j_2}$, with $(j_1,j_2)\in\mathcal J^2$.

\medskip

\noindent  {\it \underline{Extended counting function}.} For $(j_1,j_2)\in\mathcal J^2$, redefine
\begin{align*}
 \widetilde N_{k,n}({\bf q};a_{j_1},a_{j_2}) := \# \Bigl\{(\sigma,\sigma')\in\Sigma_n^2 :\ &|\sigma\wedge\sigma'|=k,\ {\bf q}(\sigma,\sigma')={\bf q}, \\
& \widetilde X_k(\sigma)\ge a_{j_1}N \ \text{ {\it (only if $j_1\ne *$)}},\\
& \widetilde X_k(\sigma')\ge a_{j_2}N \ \text{ {\it (only if $j_2\ne *$)}} \Bigr\}.
\end{align*}
That is : whenever an index equals $*$, the corresponding lower-bound constraint is  {\it dropped} --- not replaced by ``$\ge 0$''.

\medskip

\noindent  {\it \underline{The extended first-moment bound}.} We claim that Lemma \ref{lem:firstmoment}'s conclusion extends
verbatim to the whole of $\mathcal J^2$ :
\begin{equation}
\label{eq:extended_Ntilde}
\mathbb E\left[ \widetilde N_{k,n}({\bf q};a_{j_1},a_{j_2})\right] \ \le\ \exp \left\{(k+1)\log 2 +
N\left(S({\bf q})-\tfrac{a_{j_1}^2+a_{j_2}^2}{2}\right)\right\}, \qquad \forall (j_1,j_2)\in\mathcal J^2,
\end{equation}
 {\it with the convention that $a_{j_i}^2$ is replaced by $a_*^2=0$ when $j_i=*$.} Indeed, revisit
the proof of Lemma \ref{lem:firstmoment} : it bounds the number of admissible  {\it combinatorial} pairs of type ${\bf q}$
by $2^{k+1} \ee^{NS({\bf q})}$ (this count does not involve $ \widetilde X_k(\sigma)$ or $ \widetilde X_k(\sigma')$ at
all), and then multiplies by the probability that the (independent) Gaussian variables $ \widetilde X_k(\sigma)\sim\mathcal N(0,N)$ and $ \widetilde X_k(\sigma')\sim\mathcal
N(0,N)$ satisfy the two tail constraints, using $\mathbb P( \widetilde X_k(\sigma)\ge a_1 N)\le
 \ee^{-Na_1^2/2}$ and likewise for $\sigma'$. If the constraint on, say, the first coordinate is
simply dropped ($j_1=*$), the corresponding probability factor becomes $\mathbb P(\text{no
constraint})=1$ --- and $1= \ee^{-N\cdot 0^2/2}$ is  {\it exactly} the number that the Gaussian tail
bound would have produced for $a_1=0$. So the same computation as in Lemma \ref{lem:firstmoment}'s proof, with this
one factor replaced by $1$ instead of $ \ee^{-Na_1^2/2}$, gives precisely Equation \eqref{eq:extended_Ntilde}, with $a_{j_1}=0$.

\medskip

\noindent {\it \underline{Step 1. A good event}.} Recall the extended grid $\mathcal J:=\{*,0,\dots,J\}$ and
that  Equation \eqref{eq:extended_Ntilde} is valid for every $(j_1,j_2)\in\mathcal J^2$, with $\#\mathcal
J^2\le\widetilde C/\delta^2$.
 Let $\widetilde{\mathcal E}_n$ be the
event that, for every $0\le k\le\varepsilon n$, ${\bf q} \in\mathcal P_N$, $(j_1,j_2)\in\mathcal J^2$,
\begin{align*}
\text{(i)}&\quad \widetilde N_{k,n}({\bf q} ;a_{j_1},a_{j_2})\le\exp\left\{(k+1)\log2+N\left(S({\bf q} )
-\frac{a_{j_1}^2+a_{j_2}^2}{2}+\delta\right)\right\},\\
\text{(ii)}&\quad \frac{1}{2}(a_{j_1}^2+a_{j_2}^2)\ge S({\bf q} )+\delta
\ \Longrightarrow\ \widetilde N_{k,n}({\bf q} ;a_{j_1},a_{j_2})=0.
\end{align*}
By Equation \eqref{eq:extended_Ntilde} and Markov's inequality, (i) fails with probability $\le \ee^{-N\delta}\le
\ee^{-3n\delta/4}$ by Equation \eqref{eq:N-lower}. Again, by Equation \eqref{eq:extended_Ntilde}, (ii) fails with probability $\le\exp\{(k+1)\log2
-N\delta\}$, and using $k+1\le\varepsilon n+1\le\frac{n\delta}{8\log2}+1$ and Equation \eqref{eq:N-lower}, we get
$(k+1)\log2-N\delta\le\frac{n\delta}8+\log2-\frac{3n\delta}4\le-\frac{n\delta}2$, for $n$ large enough. Hence
both fail with probability $\le \ee^{-n\delta/2}$. A union bound over the $\le(n+1)^4 \widetilde C/\delta^2$ triples
$(k,{\bf q} ,(j_1,j_2))$ (using Equation \eqref{eq:PN-count} and $\#\mathcal J^2\le \widetilde C/\delta^2$) gives
\begin{equation}
\label{eq:En-bound}
\P(\widetilde{\mathcal E}_n^c)\le \frac{\widetilde Cn^4}{\delta^2}\,\ee^{-n\delta/2}, \qquad \forall \, n\ge n_3(\beta,\delta).
\end{equation}

\medskip

\noindent  {\it \underline{Step 2. Bounding $W_{k,n}(q)$ on $\widetilde{\mathcal E}_n$}.} Fix $k$ and ${\bf q}  \in \mathcal P_N$. Observe that
\[
\ee^{-\beta hN(m_1({\bf q} )+m_2({\bf q} ))} \, W_{k,n}({\bf q} )=\sum_{\substack{(\sigma,\sigma')\in\Sigma_n^2:\ |\sigma\wedge\sigma'|=k\\
{\bf q}(\sigma,\sigma')={\bf q}}} \ee^{\beta(\widetilde X_k(\sigma)+\widetilde X_k(\sigma'))}.
\]
Using the partition of $\R^2$ with the {\it cells} $(I_{j_1}\times I_{j_2})_{j_1,j_2\in\mathcal J}$, one may decompose the previous sum into sums $\Sigma_{j_1,j_2}$ given according to the {\it cell} $I_{j_1} \times I_{j_2}$ (with $j_1,j_2\in\mathcal J$) containing
$(\widetilde X_k(\sigma)/N,\widetilde X_k(\sigma')/N)$. Using (i),
\[
\Sigma_{j_1,j_2}\le\exp\left\{(k+1)\log2+N\left(S({\bf q} )-\frac{a_{j_1}^2+a_{j_2}^2}{2}+\beta(a_{j_1}+a_{j_2})
+(2\beta+1)\delta\right)\right\}, \quad \forall \, j_1,j_2\in\mathcal J.
\]
Cells with $j_1=J$ or $j_2=J$ are empty by (ii), since $a_J^2/2\ge(A+1)^2/2=A^2/2+A+\frac{1}{2}\ge
S({\bf q} )+1>S({\bf q} )+\delta$. On a surviving cell, (ii) gives $\frac{1}{2}(a_{j_1}^2+a_{j_2}^2)<S({\bf q} )+\delta$, so
by  Equation \eqref{eq:g2beta-def} applied with $S({\bf q} )+\delta$ and then Equation \eqref{eq:g2-subadd},
\[
S({\bf q} )-\frac{a_{j_1}^2+a_{j_2}^2}{2}+\beta(a_{j_1}+a_{j_2})\le g_\beta^{(2)}(S({\bf q} )+\delta)-\delta
\le g_\beta^{(2)}(S({\bf q} ))+2\beta\sqrt\delta.
\]
Summing over the $\le  \widetilde C/\delta^2$ nonempty cells, we get 
\[
W_{k,n}({\bf q} )\le\frac{\widetilde C}{\delta^2}\exp\left\{(k+1)\log2+N\left(\Xi_{\beta,h}({\bf q} )+2\beta\sqrt\delta+(2\beta+1)
\delta\right)\right\}.
\]

\medskip

\noindent  {\it \underline{Step 3. Conclusion}.} By Equation \eqref{eq:N-lower} and $\varepsilon\le\delta/(8\log2)$, we have $(k+1)\log2\le
N\delta/6+\log2$. Dividing by $N$ and taking logarithms yield
\[
\frac1N\log W_{k,n}({\bf q} )\le\Xi_{\beta,h}({\bf q} )+2\beta\sqrt\delta+(2\beta+2)\delta
+\frac{\log(2 \widetilde C\delta^{-2})}N.
\]
Finally choosing $n_3=n_3(\beta,\delta)$ large enough such that $\log(2 \widetilde C\delta^{-2})/N \le \delta$ for all $n \ge n_3$ (recall that $N\ge 3n/4$, see Equation \eqref{eq:N-lower}), and then 
$\widetilde \kappa=\widetilde \kappa(\beta)$ large enough so that 
$
2\beta\sqrt\delta+(2\beta+3)\delta
\le \widetilde \kappa\sqrt\delta
$
for all $\delta \in (0,1)$ (take $\widetilde \kappa  \ge 4\beta+3$ for example)
and
$
\widetilde \kappa \ge \widetilde C,
$
yield
\[
\frac1N\log W_{k,n}({\bf q} )\le \Xi_{\beta,h}({\bf q} )+\widetilde \kappa\sqrt\delta, \qquad \forall \, n \ge n_3.
\]
Together with
Equation \eqref{eq:En-bound}, this proves the proposition.
\end{proof}

\subsubsection*{The two-replica variational problem}

\medskip

\begin{prop}[Two-replica variational problem]
\label{prop:variational}
Fix $h>0$, $\beta>\beta_c(h)>0$, recall $m^*=m^*(h)=\tanh(\beta_c h)$ and the definition of the free energy $f(\beta,h)$ from Theorem \ref{thm:theoremB}.
\begin{itemize}
\item[(i)] We have
\[
\max_{{\bf q} \in\mathcal P}\Xi_{\beta,h}({\bf q})=2f(\beta,h),
\]
where the $\max$ is attained at a unique 
${\bf q}^* = {\bf q}^*(h) \in \mathcal P$ with
\[
q^*_{ab}=\left(\frac{1+am^*}{2}\right) \left(\frac{1+bm^*}{2}\right), \qquad \forall \, a,b \in \{-1,1\}.
\]
In particular $\rho({\bf q}^*)=(m^*)^2$.

\item[(ii)] For every $\eta\in(0,1)$,
\begin{equation}
 \label{eq:c2-def}
c_2=c_2(\beta,h,\eta):=2f(\beta,h)-\max\left\{\Xi_{\beta,h}({\bf q}):\ {\bf q}\in\mathcal P,\
|\rho({\bf q})-(m^*)^2|\ge\eta\right\}>0.
\end{equation}
\end{itemize}
\end{prop}

\medskip

\begin{proof}
{\it \underline{Proof of  (i)}.} For ${\bf q}\in\mathcal P$, write $m_i=m_i({\bf q})$ for $i=1,2$ and define $\bar m:=(m_1+m_2)/2$. By Equation \eqref{eq:g2-identity},
Equation \eqref{eq:subadd} and monotonicity of $g_\beta$, then concavity of $s$ and monotonicity of $g_\beta$,
then the definition of $\Phi_h$ (see Equation \eqref{eq:Phi}),
\begin{eqnarray}
\Xi_{\beta,h}({\bf q})&=&\beta h(m_1+m_2)+2g_\beta\left(\frac{S({\bf q})}{2}\right)
\le\beta h(m_1+m_2)+2g_\beta\left(\frac{s(m_1)+s(m_2)}{2}\right)
\label{eq:chain1}
\\
&\le&\beta h(m_1+m_2)+2g_\beta\left(s(\bar m)\right)=2\Phi_h(\beta,\bar m), 
\label{eq:chain2}
\end{eqnarray}
so $\Xi_{\beta,h}({\bf q})\le2\Phi_h(\beta,\bar m)\le2f(\beta,h)$ by Proposition \ref{prop:var}(iii). For ${\bf q}={\bf q}^*$,
$m_1=m_2=m^*$ and $S({\bf q}^*)=2s(m^*)$ (equality case of Equation \eqref{eq:subadd}), so
$\Xi_{\beta,h}({\bf q}^*)=2\Phi_h(\beta,m^*)=2f(\beta,h)$ : the maximum is $2f(\beta,h)$, attained at ${\bf q}^*$.

\medskip

\noindent {\it \underline{Uniqueness}.} If $\Xi_{\beta,h}({\bf q})=2f(\beta,h)$, all inequalities in Equations \eqref{eq:chain1} and  \eqref{eq:chain2}  are
equalities and $\Phi_h(\beta,\bar m)=f(\beta,h)$, forcing $\bar m=m^*$ (see Proposition \ref{prop:var}(iii)). Strict
monotonicity (injectivity) of $g_\beta$ applied to the second inequality forces
$\frac{s(m_1)+s(m_2)}2=s(\bar m)$, which forces $m_1=m_2=\bar m=m^*$  ($s$ is strictly concave).
Injectivity of $g_\beta$ applied to the first forces $S({\bf q})=s(m_1)+s(m_2)$, i.e.\ ${\bf q}$ is a {\it product type}, by the
equality case of Equation \eqref{eq:subadd}. With $m_1=m_2=m^*$, this means ${\bf q}={\bf q}^*$.

\medskip

\noindent
{\it \underline{Proof of  (ii)}.} 
The function ${\bf q} \mapsto (m_1({\bf q}),m_2({\bf q}),S({\bf q}))$ is clearly continuous on the compact set $\mathcal P$, hence so is
$\Xi_{\beta,h}$ by Properties \eqref{eq:gbeta-props} and Equation \eqref{eq:g2-identity}. The set $K_\eta:=\{{\bf q}\in\mathcal P:|\rho({\bf q})-(m^*)^2|\ge\eta\}$ is closed, hence compact,
and non-empty for $\eta\in(0,1)$ : the anti-diagonal type
$q_{+-}=q_{-+}=1/2$ has $\rho(q)=-1$, so
$|\rho(q)-(m^*)^2|=1+(m^*)^2>1>\eta$. Moreover ${\bf q}^*\notin K_\eta$. Therefore, $\Xi_{\beta,h}$ attains its $\max$ over $K_\eta$ at some ${\bf q}_\eta$ with
$\Xi_{\beta,h}({\bf q}_\eta)<2f(\beta,h)$ by uniqueness in (i). This is Equation \eqref{eq:c2-def}.
\end{proof}

\begin{rem}[Mechanism]
Each replica may separately reach $m^*$, but the number of pairs realizing a joint type $q$ grows
like $\ee^{NS({\bf q})}$, strictly below $\ee^{N(s(m_1)+s(m_2))}$ unless ${\bf q}$ is a {\it product type} (see Lemma \ref{lem:subadditivity}). Because the
energy a pair can carry is capped by its joint entropy ($a_1^2+a_2^2\le2S({\bf q})$) and $g_\beta^{(2)}$ is
strictly increasing in $S({\bf q})$, any non-product type is exponentially penalized (see the rate function $\Xi_{\beta,h}({\bf q})$ defined in Equation \eqref{eq:Xi-def}) and the optimal type is
the product of two $m^*$-magnetized marginals, of overlap $(m^*)^2$. 
\end{rem}

\subsection{Strategy of the proof}

From here on, we fix $h>0$ and $\beta>\beta_c(h)>0$. Let $\varphi:[-1,1]\to\R$ be continuous, $\varepsilon\in(0,1/2)$. Decompose
\begin{equation}
\label{eq:decomposition}
\mathbb E\left[\mathcal G_{\beta,h,n}^{\otimes2}\varphi(r_n)\right]=A_n(\varepsilon)+B_n(\varepsilon)+C_n(\varepsilon),
\end{equation}
where 
\begin{eqnarray*}
A_n(\varepsilon) & := & \mathbb E[\mathcal G_{\beta,h,n}^{\otimes2}\varphi(r_n)\mathbf 1_{\{q_n\ge1-\varepsilon\}}],
\\
B_n(\varepsilon) & := & \mathbb E[\mathcal G_{\beta,h,n}^{\otimes2}\varphi(r_n)\mathbf 1_{\{q_n\le\varepsilon\}}],
\\
C_n(\varepsilon) &:= & \mathbb E[\mathcal G_{\beta,h,n}^{\otimes2}\varphi(r_n)\mathbf 1_{\{\varepsilon<q_n<1-\varepsilon\}}].
\end{eqnarray*}
We show :
\begin{itemize}
\item[(i)] $C_n(\varepsilon)\to0$ as $n\to\infty$, for every $\varepsilon>0$ (Lemma \ref{lem:negligible} from
Theorem \ref{thm:theoremC}),
\item[(ii)] $A_n(\varepsilon)\to\varphi(1)(1-\beta_c/\beta)$ as $n\to\infty$, then $\varepsilon\to0$
(Lemma  \ref{lem:highoverlap}),
\item[(iii)] $B_n(\varepsilon)\to\varphi((m^*)^2)\,\beta_c/\beta$ as $n\to\infty$, then $\varepsilon\to0$
(Lemma  \ref{lem:lowoverlap}, resting on Subsection \ref{subsec:2replica}).
\end{itemize}
Summing the three contributions and using weak convergence via continuous bounded test functions
then yields Theorem \ref{thm:theorem2}, cf. Subsection \ref{subsec:conclusion2}.

\subsection{Negligibility of the intermediate region}

\medskip

\begin{lem}
\label{lem:negligible}
For every $\varepsilon\in(0,1/2)$, 
\[
\lim_{n\to\infty}\mathbb E\left[\mathcal G_{\beta,h,n}^{\otimes2}\{q_n\in(\varepsilon,
1-\varepsilon)\}\right]=0.
\]
\end{lem}

\medskip

\begin{proof}
Let $\mu_n:=\mathbb E[\mathcal G_{\beta,h,n}^{\otimes2}\{q_n\in\cdot\}]$, which is a probability measure on $[0,1]$. By
Theorem \ref{thm:theoremC}, $\mu_n$ converges weakly to $\mu:=\frac{\beta_c}\beta\delta_0+(1-\frac{\beta_c}\beta)\delta_1$. Since
$F:=[\varepsilon,1-\varepsilon]$ is closed, the portmanteau theorem gives $\limsup_n\mu_n((\varepsilon,1-\varepsilon))\le
\limsup_n\mu_n(F)\le\mu(F)=0$, since $0,1\notin F$.
\end{proof}

\subsection{High overlap I implies high overlap II}

\medskip

\begin{lem}
\label{lem:highoverlap}
For every $\varepsilon\in(0,1/2)$ and $\sigma,\sigma'\in\Sigma_n$, $q_n(\sigma,\sigma')\ge1-\varepsilon
\Rightarrow r_n(\sigma,\sigma')\ge1-2\varepsilon$. Consequently, for every continuous $\varphi:[-1,1]\to\R$,
\[
\lim_{\varepsilon\to0}\limsup_{n\to\infty}\left|A_n(\varepsilon)-\varphi(1)(1-\beta_c/\beta)\right|=0.
\]
\end{lem}

\medskip

\begin{proof} Fix $\sigma,\sigma'\in\Sigma_n$ such that $q_n(\sigma,\sigma')\ge1-\varepsilon$.
If $\sigma=\sigma'$, it is trivial. Otherwise $k=|\sigma\wedge\sigma'|\ge(1-\varepsilon)n$, so
$r_n\ge\frac{k-(n-k)}n=\frac{2k}n-1\ge1-2\varepsilon$. Let $\omega_\varphi$ be the modulus of continuity of
$\varphi$ (which is bounded, uniformly continuous on the compact $[-1,1]$). Then $|\varphi(r_n)-\varphi(1)|
\mathbf1_{\{q_n\ge1-\varepsilon\}}\le\omega_\varphi(2\varepsilon)$, so, integrating yields
$\left|A_n(\varepsilon)-\varphi(1)\,\mathbb E[\mathcal G_{\beta,h,n}^{\otimes2}\{q_n\ge1-\varepsilon\}]\right|\le
\omega_\varphi(2\varepsilon)$. Since $[1-\varepsilon,1]$, viewed as a subset of the state space
$[0,1]$, is a continuity set of
$\mu:=\frac{\beta_c}\beta\delta_0+(1-\frac{\beta_c}\beta)\delta_1$,\footnote{ \label{footnote:continuity-set}
Relative to $\mathbb R$ one has $\partial_{\mathbb R}[1-\varepsilon,1]=\{1-\varepsilon,1\}$, and
$\mu(\{1\})=1-\beta_c/\beta \neq 0$, for $\beta>\beta_c$; so the boundary must be understood relative to the
 state space $[0,1]$ on which the $\mu_n:=\E[\mathcal G_{\beta,h,n}^{\otimes2}\{q_n\in\cdot\}]$'s and $\mu$
are defined. Since $1$ is the right endpoint of $[0,1]$, it belongs to the interior of
$[1-\varepsilon,1]$ relative to $[0,1]$, so
$\partial_{[0,1]}[1-\varepsilon,1]=\{1-\varepsilon\}$. Since $\varepsilon\in(0,1/2)$, one has
$1-\varepsilon\in(1/2,1)$, hence $1-\varepsilon\notin\{0,1\}$ and $\mu(\partial_{[0,1]}[1-\varepsilon,1])=\mu(\{1-\varepsilon\})=0$, as
required.}
 Theorem \ref{thm:theoremC} gives $\mathbb E[
\mathcal  G_{\beta,h,n}^{\otimes2}\{q_n\ge1-\varepsilon\}]\to1-\beta_c/\beta$ when $n \to \infty$, and $\omega_\varphi(2\varepsilon)\to0$, when $\varepsilon\to0$, concludes the proof.
\end{proof}

\subsection{Low overlap I implies overlap II close to $(m^*)^2$}

This is the key step : no conditioning on the subtrees at the branching point is used. Instead, an
upper bound for the two-replica partition function restricted to {\it bad} types is compared, on an
event of overwhelming probability, with the lower bound of Proposition \ref{prop:lower} for $Z_{\beta,h,n}^2$.

\medskip

\begin{lem}
\label{lem:lowoverlap}
For every $\eta>0$,
\[
\lim_{\varepsilon\to0}\limsup_{n\to\infty}\mathbb E\left[\mathcal G_{\beta,h,n}^{\otimes2}\{|r_n-(m^*)^2|>\eta \, ; \, q_n\le\varepsilon\}\right]=0.
\]
More precisely, we have the stronger result that, for every $\eta>0$, there exists $\varepsilon_0=\varepsilon_0(\beta,h,\eta)>0$
such that the $\limsup_n$ vanishes for every $\varepsilon\in(0,\varepsilon_0)$.
\end{lem}

\medskip

\begin{proof}
Since the event $\{|r_n-(m^*)^2|>\eta\,;\,
q_n\le\varepsilon\}$ is non-increasing in  $\eta$, it is enough to prove the claim for $\eta \in(0,1)$. Fix such an $\eta \in(0,1) $.
 Note that $\eta/2\in(0,1)$, so
that Proposition \ref{prop:variational} (ii) applies to it.
Let $c_2=c_2(\beta,h,\eta/2)>0$ be the gap given by Proposition \ref{prop:variational}(ii),
$\widetilde \kappa=\widetilde \kappa(\beta)<\infty$ the constant  from Proposition \ref{prop:unifbound}, $\Theta_0=\Theta_0(h)$ the constant from Lemma \ref{lem:head} and
\begin{equation}
\label{eq:C_tilde}
\widetilde C=\widetilde C(\beta,h):=2\beta\Theta_0+\sup_{{\bf q} \in\mathcal P}|\Xi_{\beta,h}({\bf q})|<\infty,
\end{equation}
since $\Xi_{\beta,h}$ is continuous on the compact $\mathcal P$. Choose $\delta\in(0,1)$ small enough such that
\begin{equation}
\label{eq:delta-choice}
\widetilde \kappa\sqrt\delta+2\beta\delta+2\delta\le \frac{c_2}{8}, 
\end{equation}
and then $\varepsilon_0\in(0,1)$ so that
\begin{equation}
\label{eq:eps0-choice}
\varepsilon_0\le\min\left\{\frac18,\ \frac{\delta}{8\log2},\ \frac\eta{8},\ \frac{c_2}{8\widetilde C}\right\}.
\end{equation}

Fix $\varepsilon\in(0,\varepsilon_0)$ and observe that all constants depend on $(\beta,h,\eta)$ only.

\medskip
\medskip

\noindent {\it \underline{Step 1. Reduction to types}.} Let $(\sigma,\sigma')$ satisfy $q_n\le\varepsilon$ and
$|r_n-(m^*)^2|>\eta$. Then $\sigma\ne\sigma'$, $k:=nq_n\le\varepsilon n$, and ${\bf q}={\bf q}(\sigma,\sigma')\in
\mathcal P_{n-k-1}$ is well defined for $n\ge3$ (as $\varepsilon\le1/8$). By Equation \eqref{eq:overlap-type},
\[
|\rho({\bf q})-(m^*)^2|\ge|r_n-(m^*)^2|-\frac{2(k+1)}n\ge\eta-2\varepsilon-\frac2n\ge\frac\eta2,
\]
for $n\ge n_4(\eta):=\lceil8/\eta\rceil$, using here $\varepsilon\le\eta/8$. Hence, writing, for ${\bf q} \in \mathcal P_{N}$,
\[
Z^{(2)}_{k,n}({\bf q}):=\sum_{\substack{(\sigma,\sigma')\in\Sigma_n^2:\ |\sigma\wedge\sigma'|=k\\
{\bf q}(\sigma,\sigma')={\bf q}}}
\ee^{\beta H_n(\sigma,h)+\beta H_n(\sigma',h)},
\]
we get, for $n\ge\max\{3,n_4(\eta)\}$,
\begin{equation}
\label{eq:step1}
\mathcal G_{\beta,h,n}^{\otimes2}\{|r_n-(m^*)^2|>\eta \, ; \, q_n\le\varepsilon\}\le\frac1{Z_{\beta,h,n}^2} \, 
\sum_{0\le k\le\varepsilon n}\ \sum_{\substack{{\bf q}\in\mathcal P_{n-k-1}\\|\rho({\bf q})-(m^*)^2|\ge\eta/2}}
Z^{(2)}_{k,n}({\bf q}). 
\end{equation}

\medskip

\noindent  {\it \underline{Step 2. A good event}.} Let $\widetilde{\mathcal H}_n:=\widetilde{\mathcal M}_n\cap \widetilde{\mathcal K}_n\cap\{\frac1n\log Z_{\beta,h,n}\ge
f(\beta,h)-\delta\}$, with $\widetilde{\mathcal M}_n$ from Proposition \ref{prop:unifbound} (with this $\delta,\varepsilon$), and $\widetilde{\mathcal K}_n$ from Lemma \ref{lem:head}. By Proposition \ref{prop:unifbound}, Lemma \ref{lem:head} and
Proposition \ref{prop:lower}, there is $n_5=n_5(\beta,h,\delta)$ such that for all $n\ge n_5$,
\begin{equation}
 \label{eq:Hn-bound}
\P(\widetilde{\mathcal H}_n^c)\le\frac{\widetilde \kappa n^4}{\delta^2}\ee^{-n\delta/2}+ n \, \ee^{-\sqrt{2\log2}\,\delta n}
+\ee^{-n\delta^2/(8\beta^2)}\xrightarrow[n\to\infty]{}0.
\end{equation}

\medskip

\noindent  {\it \underline{Step 3. Bounding $Z^{(2)}_{k,n}({\bf q})$}.} Fix $k\le\varepsilon n$, ${\bf q}\in \mathcal P_{n-k-1}$. Splitting each
replica's Hamiltonian at $k+1$ and bounding both heads, namely $H_{k+1}(\sigma|_{k+1},h)$ and $H_{k+1}(\sigma'|_{k+1},h)$, on $ \widetilde{\mathcal K}_n$ by $\Theta_0(k+1)+\delta n$,
Equation \eqref{eq:Wkn-def} gives, on $ \widetilde{\mathcal K}_n$, 
\[
Z^{(2)}_{k,n}({\bf q})\le \ee^{2\beta(\Theta_0(k+1)+\delta n)} \, 
W_{k,n}({\bf q}).
\]
On $\widetilde{\mathcal M}_n$, Proposition \ref{prop:unifbound} gives 
\[
W_{k,n}({\bf q})\le\exp\{(n-k-1)[\Xi_{\beta,h}({\bf q})
+\widetilde \kappa\sqrt\delta]\}.
\]
Therefore, on $\widetilde{\mathcal H}_n$, using $k+1\le\varepsilon n+1$ and
$|\Xi_{\beta,h}({\bf q})|+2\beta\Theta_0\le \widetilde C$ (by Equation \eqref{eq:C_tilde}), we have
\begin{eqnarray*}
\frac1n\log Z^{(2)}_{k,n}({\bf q}) & \le&  \Xi_{\beta,h}({\bf q}) + \widetilde \kappa\sqrt\delta+2\beta\delta + \frac{k+1}{n}\left( |\Xi_{\beta,h}({\bf q})|+2\beta\Theta_0 \right) 
\\
&\le&\Xi_{\beta,h}({\bf q}) + \widetilde \kappa\sqrt\delta+2\beta\delta+\varepsilon \widetilde C+\frac{\widetilde C}n.
\end{eqnarray*}
On the range of ${\bf q}$ appearing in Equation \eqref{eq:step1}, $\Xi_{\beta,h}({\bf q})\le2f(\beta,h)-c_2$ by
Proposition \ref{prop:variational}(ii); with $\widetilde \kappa\sqrt\delta+2\beta\delta\le c_2/8$
(by Equation \eqref{eq:delta-choice}) and $\varepsilon \widetilde C\le\varepsilon_0\widetilde C\le c_2/8$ (see Equation \eqref{eq:eps0-choice}), we get, on $\widetilde{\mathcal H}_n$,
\begin{equation}
\label{eq:step3-bound}
\frac1n\log Z^{(2)}_{k,n}({\bf q})\le 2f(\beta,h)-\frac{c_2}2+\frac{\widetilde C}n. 
\end{equation}

\medskip

\noindent  {\it \underline{Step 4. Conclusion}.} By Equation \eqref{eq:PN-count} the number of pairs $(k,{\bf q})$ in Equation \eqref{eq:step1} is
at most $(n+1)^4$, and $Z_{\beta,h,n}^2\ge \ee^{2n(f(\beta,h)-\delta)}$ on $\widetilde{\mathcal H}_n$. Plugging
Equation \eqref{eq:step3-bound} into Equation \eqref{eq:step1} yields, on $\widetilde{\mathcal H}_n$ and for $n\ge\max\{3,n_4,n_5\}$,
\[
\mathcal G_{\beta,h,n}^{\otimes2}\{|r_n-(m^*)^2|>\eta \, ; \,q_n\le\varepsilon\}\le(n+1)^4\ee^{\widetilde C}
\exp\left\{-n\left(\frac{c_2}2-2\delta\right)\right\}\le(n+1)^4\ee^{\widetilde C}\ee^{-nc_2/4},
\]
using $2\delta\le c_2/4$ (see Equation \eqref{eq:delta-choice}). Splitting on $\widetilde{\mathcal H}_n$ and using Equation \eqref{eq:Hn-bound} implies
\[
\mathbb E\left[\mathcal G_{\beta,h,n}^{\otimes2}\{|r_n-(m^*)^2|>\eta \, ; \,q_n\le\varepsilon\}\right]\le
(n+1)^4\ee^{\widetilde C}\ee^{-nc_2/4}+\P(\widetilde{\mathcal H}_n^c)\xrightarrow[n\to\infty]{}0.
\]
As $\varepsilon\in(0,\varepsilon_0)$ was arbitrary, this proves the lemma.
\end{proof}

\subsection{Proof of Theorem \ref{thm:theorem2}}
\label{subsec:conclusion2}

Let $\varphi:[-1,1]\to\R$ be continuous, with modulus of continuity $\omega_\varphi$. Considering
Equation \eqref{eq:decomposition} we have, by Lemma \ref{lem:negligible}, that $|C_n(\varepsilon)|\le\|\varphi\|_\infty
\,\mathbb E[\mathcal G_{\beta,h,n}^{\otimes2}\{q_n\in(\varepsilon,1-\varepsilon)\}]\to0$ when $n \to \infty$ for every $\varepsilon\in(0,1/2)$ and, by
Lemma \ref{lem:highoverlap}, that $\lim_{\varepsilon\to0}\limsup_n|A_n(\varepsilon)-\varphi(1)(1-\beta_c/\beta)|=0$.

\medskip

For $B_n(\varepsilon)$ : 
let $\eta\in(0,1)$. Splitting $\{q_n\le\varepsilon\}$ according to whether
$|r_n-(m^*)^2|\le\eta$, we get
\[
\left|B_n(\varepsilon)-\varphi\big((m^*)^2\big)\,\mathbb E[\mathcal G_{\beta,h,n}^{\otimes2}\{q_n\le\varepsilon\}]\right|
\le\omega_\varphi(\eta)+2\|\varphi\|_\infty\,\mathbb E\left[\mathcal  G_{\beta,h,n}^{\otimes2}
\{|r_n-(m^*)^2|>\eta,\,q_n\le\varepsilon\}\right].
\]
Since $[0,\varepsilon]$, viewed as a subset of the state space
$[0,1]$, is a continuity set of $\frac{\beta_c}\beta\delta_0+(1-\frac{\beta_c}\beta)\delta_1$ (see Footnote \ref{footnote:continuity-set}),
Theorem \ref{thm:theoremC}  gives $\mathbb E[\mathcal G_{\beta,h,n}^{\otimes2}\{q_n\le\varepsilon\}]\to\beta_c/\beta$. Letting $n\to
\infty$, then $\varepsilon\to0$ and using Lemma  \ref{lem:lowoverlap}, yield
\[
\limsup_{\varepsilon\to0}\limsup_{n\to\infty}\left|B_n(\varepsilon)-\varphi\big((m^*)^2\big)\,\frac{\beta_c}
\beta\right|\le\omega_\varphi(\eta),
\]
and, $\eta\in(0,1)$ being arbitrary with $\omega_\varphi(\eta)\to_{\eta \to 0}0$, the left-hand side vanishes.

Summing the three contributions in Equation \eqref{eq:decomposition} and letting first $n\to\infty$, then
$\varepsilon\to0$,
\[
\lim_{n\to\infty}\mathbb E\left[\mathcal  G_{\beta,h,n}^{\otimes2}\varphi(r_n)\right]=\frac{\beta_c}\beta
\, \varphi\big((m^*)^2\big) +\left(1-\frac{\beta_c}\beta\right)\varphi(1).
\]
As this holds for every continuous $\varphi:[-1,1]\to\R$ and all measures involved are supported on
the compact $[-1,1]$, the sequence $\mathbb E[\mathcal  G_{\beta,h,n}^{\otimes2}\{r_n\in\cdot\}]$ converges
weakly to
\[
\frac{\beta_c(h)}\beta\,\delta_{(m^*)^2}+\left(1-\frac{\beta_c(h)}\beta\right)\delta_1,
\]
which is the statement of Theorem \ref{thm:theorem2}. \hfill$\square$

\appendix

\section{APPENDIX}

\subsection{Many-to-one lemma}
\label{subsec:appendix-1}
\medskip

The following identity is a standard tool in the study of branching
random walks, usually referred to as the {\it many-to-one lemma}. It goes
back to the work of Kahane and Peyri\`ere \cite{kahanepeyriere76} on
Mandelbrot's multiplicative cascades. See Shi \cite[Chapter 1]{shi2015} and
Zeitouni \cite{zeitouni20notes} for a modern, self-contained proof in the i.i.d.\ case, and
Biggins and Kyprianou \cite{bigginskyprianou2004} for the general change-of-measure formulation ---
adapted below to the case of independent, but non-identically distributed,
displacements relevant to $H_n(\cdot,h)$.

\medskip

\begin{lem}[Many-to-one identity]
\label{lem:m21}
For all $(t,\theta)\in\R^2$ and all $n\ge1$,
\[
\E\left[\ \sum_{\sigma\in\Sigma_n}\ee^{\langle(t,\theta),{\bf V}_n(\sigma)\rangle}\right]
=\ee^{n\Lambda(t,\theta)}.
\]
\end{lem}

\medskip

\begin{proof}
Let $W_n(t,\theta):=\sum_{\sigma\in\Sigma_n}\ee^{t X_n(\sigma)+\theta\sum_{i\le n}\sigma_i}$
and let $\mathcal F_n$ be the $\sigma$-field generated by
$(U(\sigma))_{|\sigma|\le n}$. Splitting each $\sigma\in\Sigma_{n+1}$ as
$\sigma=\epsilon$ with $v\in\Sigma_n$ and $\epsilon\in\{-1,+1\}$,
\[
W_{n+1}(t,\theta)=\sum_{v\in\Sigma_n}\ee^{tX_n(v)+\theta\sum_{i\le n}v_i}
\sum_{\epsilon=\pm1}\ee^{tU(v \epsilon)+\theta\epsilon}.
\]
The variables $(U(v \epsilon))_{v\in\Sigma_n,\epsilon=\pm1}$ are independent of
$\mathcal F_n$ and i.i.d.\ standard Gaussian random variables, hence
\[
\E\left[\sum_{\epsilon=\pm1}\ee^{tU(v\epsilon)+\theta\epsilon} \; \Big| \;  \mathcal F_n\right]
=\ee^{t^2/2}\left(\ee^{\theta}+\ee^{-\theta}\right)=\ee^{\Lambda(t,\theta)}, \qquad \forall \, v\in\Sigma_n,
\]
so that $\E[W_{n+1}(t,\theta) \, \vert \, \mathcal F_n ]=\ee^{\Lambda(t,\theta)}W_n(t,\theta)$ and $\E[W_{n+1}(t,\theta)]=\ee^{\Lambda(t,\theta)}\E[W_n(t,\theta)]$. Since $W_0=1$, the claim
follows by induction.
\end{proof}

\subsection{Subadditivity of entropy for a pair of empirical types}
\label{app:subadditivity}

We recall here, without proof, the classical information-theoretic
facts used in Equation \eqref{eq:subadd}.
Let ${\bf q} = (q_{ab})_{a,b \in \{-1,1\}} \in \mathcal{P}$ be a probability
vector on $\{-1,1\}^2$, with marginals
\[
p_a := \sum_{b \in \{-1,1\}} q_{ab}, \qquad
r_b := \sum_{a \in \{-1,1\}} q_{ab}, \qquad a,b \in \{-1,1\},
\]
and recall the notation
\[
m_1({\bf q}) := \sum_{a,b} a\, q_{ab} = p_{+1} - p_{-1}, \qquad
m_2({\bf q}) := \sum_{a,b} b\, q_{ab} = r_{+1} - r_{-1}, \qquad
S({\bf q}) := -\sum_{a,b} q_{ab} \log q_{ab}.
\]
Since $p_{\pm 1} = \frac{1 \pm m_1({\bf q})}{2}$ and $r_{\pm 1} = \frac{1 \pm m_2({\bf q})}{2}$,
the Shannon entropies of the two marginals of $q$ are exactly
$s(m_1(q))$ and $s(m_2(q))$, where $s$ is the binary entropy function
of Equation \eqref{eq:binary-entropy}.

\medskip

\begin{lem}[Subadditivity of entropy]
\label{lem:subadditivity}
For every ${\bf q} \in \mathcal{P}$,
\[
S({\bf q}) \;\le\; s(m_1({\bf q})) + s(m_2({\bf q})),
\]
with equality if and only if ${\bf q}$ is of product form, i.e.
\[
q_{ab} = \left(\frac{1 + a\, m_1({\bf q})}{2}\right)\left(\frac{1 + b\, m_2({\bf q})}{2}\right),
\qquad \forall\, a, b \in \{-1,1\}.
\]
Equivalently, equality holds if and only if the two coordinates are
independent under ${\bf q}$, in which case $\rho({\bf q}) = m_1({\bf q})\, m_2({\bf q})$.
\end{lem}

\medskip

\begin{rem}
Lemma  \ref{lem:subadditivity} is the two-letter instance, for a
discrete random pair $(X,Y)$ with joint law ${\bf q}$ and marginal laws
$P, R$, of the general subadditivity property of Shannon entropy :
\[
H(X,Y) \le H(X) + H(Y),
\]
where $H(\cdot)$ denotes the entropy; equivalently the non-negativity of the mutual information
$I(X;Y) := H(X) + H(Y) - H(X,Y) \ge 0$, itself a consequence of the
non-negativity of the Kullback--Leibler divergence often written $D({\bf q} \,\|\, P
\otimes R) \ge 0$ (Gibbs' inequality). See \cite[Theorem 2.6.6]{CoverThomas06}
 for the general statement and proof.
\end{rem}

\subsection{Proofs of Theorem \ref{thm:theoremA}, Theorem \ref{thm:theoremB} and Theorem \ref{thm:theoremC}}
\label{subsec:proofABC}

By the key algebraic identity of Subsection \ref{subsec:model}, $(H_n(\sigma,h))_{\sigma\in\Sigma_n,\,n\ge1}$
is the branching random walk with reproduction point process given by
\[
  \mathcal L \;=\; \delta_{X_+}+\delta_{X_-},
  \qquad X_+\sim\mathcal N(h,1),\quad X_-\sim\mathcal N(-h,1), \quad \text{with $X_+$, $X_-$ independent.}
\]
We record once and for all the three features of $\mathcal L$ used below :

\begin{enumerate}[label=\textup{(F\arabic*)},leftmargin=2.6em,itemsep=1pt,topsep=3pt]
\item the number of children is deterministic, equal to $2$. In particular, the branching
  random walk is supercritical, $\E[\#\mathcal L]=2>1$.
\item For every $t\in\R$,
  $\ \E\left[\sum_i \ee^{tX_i}\right]=\ee^{t^2/2}\left(\ee^{th}+\ee^{-th}\right)=\ee^{\psi_h(t)}<\infty$,
  with $\psi_h$ as in Theorem \ref{thm:theoremA}. More generally, every polynomial in $(X_+,X_-)$ times
  $\ee^{tX_\pm}$ is integrable.
\item $X_+$ and $X_-$ have densities, so $\mathcal L$ is non-lattice.
\end{enumerate}
Every supercriticality, integrability and non-lattice assumption required by \cite{biggins76}, \cite{chauvinrouault97}
and \cite{Mallein2018} follows directly from (F1)--(F3). More precisely, Properties (F1) and (F2) are used throughout Theorems \ref{thm:theoremA}, \ref{thm:theoremB} and \ref{thm:theoremC} below while Property
(F3) plays no role for Theorems \ref{thm:theoremA} and \ref{thm:theoremB}, but is required for Theorem \ref{thm:theoremC}, whose proof rests on
Madaule's convergence of the extremal process \cite{madaule2017} and hence on the classical non-lattice
condition.
The only quantitative input is the elementary
lemma below, which locates $\beta_c$ and, in (iv), computes the variance parameter needed
in \cite{Mallein2018}.

\medskip

\begin{lem}
\label{lem:Psi-h}
Fix $h>0$, recall $\psi_h(t)=\tfrac{t^2}{2}+\log2+\log(\cosh(ht))$ and set
$\Delta_h(t):=t\psi_h'(t)-\psi_h(t)$, for $t\ge0$. Then :
\begin{enumerate}[label=\textup{(\roman*)},leftmargin=2.3em,itemsep=1pt,topsep=3pt]
\item $\psi_h'(t)=t+h\tanh(ht)$, $\ 1\le\psi_h''(t)=1+h^2\cosh^{-2}(ht)\le 1+h^2$, and
  $\Delta_h'(t)=t\,\psi_h''(t)>0$.
\item $\Delta_h(0)=-\log2$ and $\Delta_h(t)\ge \tfrac{t^2}{2}-\log2$. Consequently
  $\Delta_h$ is a strictly increasing bijection from $[0,\infty)$ onto $[-\log2,\infty)$
  and vanishes
 at a unique point
$\beta_c=\beta_c(h)\in(0,\infty)$, characterized by 
$$
\beta_c\psi_h'(\beta_c)-\psi_h(\beta_c)=0,
$$
i.e. 
Equation \eqref{eq:beta_c}. Moreover, 
  \[
    \Delta_h(\beta)<0 \ \text{ for }\beta<\beta_c,
    \qquad \Delta_h(\beta)>0 \ \text{ for }\beta>\beta_c.
  \]
\item The function $\Psi_h(t):=\psi_h(t)/t$, for $t>0$, satisfies $\Psi_h'(t)=\Delta_h(t)/t^2$. Hence
  $\Psi_h$ is strictly decreasing on $(0,\beta_c]$, strictly increasing on $[\beta_c,\infty)$, and
  \[
    \inf_{t>0}\Psi_h(t)=\Psi_h(\beta_c)=\psi_h'(\beta_c)=\beta_c+h\tanh(\beta_c h)=\gamma_{\max}(h).
  \]
\item $0<\sigma_h^2:=\beta_c^{\,2}\,\psi_h''(\beta_c)\le \beta_c^{\,2}(1+h^2)<\infty$.
\end{enumerate}
\end{lem}

\medskip

\begin{proof}
(i) is a direct computation, and gives $\Delta_h'(t)=t\psi_h''(t)>0$, for $t>0$. $\Delta_h(0)=-\log2$ is trivial, hence
$\Delta_h(t)=-\log2+\int_0^t s\,\psi_h''(s)\,ds\ge \tfrac{t^2}{2}-\log2\to\infty$, when $t \to \infty$, which is
(ii). For (iii), $\Psi_h'(t)=\left(t\psi_h'(t)-\psi_h(t)\right)/t^2=\Delta_h(t)/t^2$ has the
sign of $\Delta_h(t)$ and $\Delta_h(\beta_c)=0$ reads $\psi_h(\beta_c)=\beta_c\psi_h'(\beta_c)$, i.e.\
$\Psi_h(\beta_c)=\psi_h'(\beta_c)$. Finally (iv) follows from the bounds in (i).
\end{proof}

\medskip

We now turn to the proofs of Theorems \ref{thm:theoremA}--\ref{thm:theoremC} themselves, referring to the
notation and to Lemma \ref{lem:Psi-h} introduced above.

\subsubsection*{Proof of Theorem \ref{thm:theoremA}}

Biggins \cite{biggins76} proved that a supercritical branching random walk whose reproduction law has a
finite Laplace transform $m(t):=\E\left[\sum_i \ee^{tX_i}\right]<\infty$, on $(0,\infty)$, satisfies
\begin{equation}
\label{eq:app-1}
  \lim_{n\to\infty}\frac1n\max_{\sigma\in\Sigma_n}H_n(\sigma,h)
  \;=\;\inf_{t>0}\frac{\log m(t)}{t},\qquad\text{a.s.}
\end{equation}
Supercriticality is (F1) and the finiteness of $m$ is (F2), which moreover identifies
$\log m=\psi_h$. No further assumption is needed at this order (the non-lattice property
(F3), required in Theorem \ref{thm:theoremC} below, plays no role here). By Lemma \ref{lem:Psi-h}(iii) the infimum
in Equation \eqref{eq:app-1} is a minimum, attained at the unique point $\beta_c$, and equals $\gamma_{\max}(h)$.
This is the almost sure statement in Equation \eqref{eq:gammamax}.

For the $L^1$ convergence, view  $F_n:=n^{-1}\max_{\sigma\in\Sigma_n}H_n(\sigma,h)$ as a
function of the standard Gaussian vector $(U(v))_{1\le|v|\le n}$, so that it is $n^{-1/2}$-Lipschitz. Gaussian concentration \cite[Theorem
2.2.4]{talagrand2003} then gives $\P(|F_n-\E [ F_n]|>u)\le 2\ee^{-nu^2/2}$, for all $u>0$, so that $(F_n)_{n\ge1}$
is uniformly integrable, and $L^1$ convergence follows from the almost sure convergence.
\hfill$\square$

\subsubsection*{Proof of Theorem \ref{thm:theoremB}}

\medskip

\noindent {\it \underline{Dictionary}.} In \cite{chauvinrouault97}, the branching random walk is minimised : particles sit at positions
$X_u$, the partition function is $Z_n(\beta):=\sum_{|u|=n}\ee^{-\beta X_u}$, the free energy  is $F_n(\beta):=-(n\beta)^{-1}\log Z_n(\beta)$, and, $\lambda$ denoting the
intensity measure of the offspring point process, $m(\beta):=\int_\R \ee^{-\beta x}\lambda(dx)$
and $\ell(\beta):=\log m(\beta)$. Taking $X_u=-H_n(\sigma,h)$ we get
$Z_n(\beta)=Z_{\beta,h,n}$, hence
\begin{equation}
\label{eq:app-2}
  f_n(\beta,h)=-\beta\,F_n(\beta), \qquad\text{and, by (F2),}\qquad \ell=\psi_h .
\end{equation}

\medskip

\noindent {\it \underline{Hypotheses}.} Theorem 1 of \cite{chauvinrouault97}  is proved under the standing assumptions of \cite[Section 2]{chauvinrouault97} --- the offspring population is a.s.\ non-zero and has mean $>1$ --- together with
the first-moment assumption
\[
  (H_0)\qquad \quad m(\beta)<\infty, \qquad \forall \,  \beta\in\R .
\]
The standing assumptions are (F1), and $(H_0)$ is (F2). No further hypothesis is needed.

\medskip

\noindent {\it \underline{The critical constants}.} The two constants $\widetilde\beta_c<\beta_c$ of \cite[Equation (1.1)]{chauvinrouault97}
are defined through the sign of $\beta \ell'(\beta)-\ell(\beta)$, which by Equation \eqref{eq:app-2} is exactly
$\Delta_h(\beta)$. By Lemma \ref{lem:Psi-h}(ii), $\Delta_h<0$ on $[0,\beta_c)$ and $\Delta_h>0$ on
$(\beta_c,\infty)$ : the constant $\beta_c$ of \cite{chauvinrouault97} is therefore the $\beta_c$ of Equation \eqref{eq:beta_c}, and
$\widetilde\beta_c<0$, consistently with the centering $\ell'(0)=\psi_h'(0)=0$, adopted in \cite{chauvinrouault97},
which here holds automatically.

\medskip

\noindent {\it \underline{Conclusion}.} Equations (1.3) and (1.4) of \cite[Theorem 1]{chauvinrouault97}
 read
$F_n(\beta)\to-\ell(\beta)/\beta$ a.s., for $\beta<\beta_c$, and
$F_n(\beta)\to-\ell(\beta_c)/\beta_c=-\ell'(\beta_c)$ a.s., for $\beta\ge\beta_c$. By Equation \eqref{eq:app-2} and Lemma \ref{lem:Psi-h}(iii), when $n \to \infty$,
\[
  f_n(\beta,h)\;\longrightarrow\;
  \begin{cases}
    \psi_h(\beta)=\log2+\dfrac{\beta^2}{2}+\log\cosh(\beta h), & {\rm if  \ } \beta<\beta_c,\\[6pt]
    \beta\,\psi_h'(\beta_c)=\beta\,\gamma_{\max}(h), &  {\rm if  \ }  \beta\ge\beta_c,
  \end{cases}
  \qquad\text{a.s.},
\]
the two expressions agreeing at $\beta=\beta_c$ since $\Delta_h(\beta_c)=0$. Finally, the $L^1$-convergence follows by concentration of measure exactly as in the proof of Proposition \ref{prop:lower},
since $n^{-1}\log Z_{\beta,h,n}$ is $\beta n^{-1/2}$-Lipschitz as a function of the standard Gaussian vector
$(U(v))_{1\le|v|\le n}$. \hfill$\square$

\subsubsection*{Proof of Theorem \ref{thm:theoremC}}

Following the strategy sketched in Subsection \ref{subsec:brw-literature}, consider the tilted branching random walk
\[
  V_n(\sigma,h):=\psi_h(\beta_c)\,n-\beta_c\,H_n(\sigma,h),\qquad \forall \, \sigma\in\Sigma_n, \  \forall \, n\ge1,
\]
whose reproduction point process is $\widetilde{\mathcal L}=\delta_{Y_+}+\delta_{Y_-}$ with
$Y_\pm:=\psi_h(\beta_c)-\beta_c X_\pm$. Since $H_n$ is maximised while $V_n$ is minimised, this is
the orientation used in \cite{Mallein2018}. We check Assumptions (1.1)--(1.4) of \cite{Mallein2018}, together
with the non-lattice assumption.

\medskip

\noindent {\it (1.1) \underline{Supercriticality}.} Immediate from (F1).

\medskip

\noindent
{\it (1.2) \underline{Boundary case}.} By (F2), $\E\left[\sum_i \ee^{\beta_c X_i}\right]=\ee^{\psi_h(\beta_c)}$ and
$\E\left[\sum_i X_i\ee^{\beta_c X_i}\right]=\psi_h'(\beta_c)\ee^{\psi_h(\beta_c)}$, hence
\[
  \E\left[\sum_i \ee^{-Y_i}\right]=\ee^{-\psi_h(\beta_c)}\,\ee^{\psi_h(\beta_c)}=1,
\]
and
\[
  \E\left[\sum_i Y_i\ee^{-Y_i}\right]
  =\ee^{-\psi_h(\beta_c)}\left(\psi_h(\beta_c)\,\ee^{\psi_h(\beta_c)}-\beta_c\,\psi_h'(\beta_c)\ee^{\psi_h(\beta_c)}\right)
  =\psi_h(\beta_c)-\beta_c\psi_h'(\beta_c)=0,
\]
by Lemma \ref{lem:Psi-h}(ii).

\medskip

\noindent
{\it (1.3) \underline{Finite variance}.} Let
$\widetilde\Lambda(\theta):=\log\E\left[\sum_i \ee^{\theta Y_i}\right]
=\theta\psi_h(\beta_c)+\psi_h(-\theta\beta_c)$, finite, for all $\theta\in\R$, by (F2). By Assumption (1.2) of \cite{Mallein2018},
$\widetilde\Lambda(-1)=0$ and $\widetilde\Lambda'(-1)=0$, so that
\[
  \sigma^2:=\E\left[\sum_i Y_i^2\ee^{-Y_i}\right]=\widetilde\Lambda''(-1)
  =\beta_c^{\,2}\psi_h''(\beta_c)=\sigma_h^2\in(0,\infty),
\]
by Lemma \ref{lem:Psi-h}(iv).

\medskip

\noindent
{\it (1.4) \underline{Integrability}.} Since there are only two children and $Y_\pm$ are affine images
of Gaussian variables, $\log_+\left(\sum_i(1+(Y_i)_+)\ee^{-Y_i}\right)\le
C\left(1+|X_+|+|X_-|\right)$ for some $C=C(h,\beta_c)<\infty$, so the left-hand side of the equation in Assumption (1.4) of \cite{Mallein2018} is at
most
\[
  C^2\,\ee^{-\psi_h(\beta_c)}\,
  \E\left[\left(\ee^{\beta_c X_+}+\ee^{\beta_c X_-}\right)\left(1+|X_+|+|X_-|\right)^2\right]<\infty,
\]
by (F2). (In fact $\sum_i\ee^{-Y_i}$ has moments of every order, far more than Assumption (1.4) in \cite{Mallein2018}  requires.)

\medskip

\noindent
{\it \underline{Non-lattice}.} Each $Y_\pm$ is a non-degenerate affine image ($\beta_c\neq0$) of a Gaussian
variable, hence has a density: this is (F3).

\medskip

All assumptions of  \cite{Mallein2018}, being satisfied, we may apply \cite[Theorem 4.3]{Mallein2018} to $V_n(\cdot,h)$.
Write $\widetilde {\mathcal G}_{\beta',h,n}(\sigma):=\ee^{-\beta'V_n(\sigma,h)}/\widetilde
Z_{\beta',h,n}$, with obvious notation for the partition function $\widetilde
Z_{\beta',h,n}$. The deterministic shift $\psi_h(\beta_c)n$ cancels in the normalisation, so that
\[
\mathcal  G_{\beta,h,n}(\cdot)=\widetilde { \mathcal G}_{\beta/\beta_c,h,n}(\cdot),\qquad \forall \, \beta>0,
\]
and the genealogical overlap $q_n$ is unchanged by the tilting, since $|\sigma\wedge\sigma'|$
depends on the tree only. Consequently
$\E\left[ \mathcal G^{\otimes2}_{\beta,h,n}\{q_n\in\cdot\}\right]=\E\left[\omega_{n,\beta/\beta_c}(\cdot)\right]$, with
$\omega_{n,\beta'}$ the overlap measure of \cite[Equation (4.3)]{Mallein2018} (the centering terms $m_n$
cancel in the ratio). Our low-temperature condition $\beta>\beta_c$ is exactly the condition
$\beta'=\beta/\beta_c>1$ of \cite[Theorem 4.3]{Mallein2018}, which yields
\[
  \omega_{n,\beta'}\;\xrightarrow[n\to\infty]{\rm (law)}\;(1-\pi_{\beta'})\,\delta_0+\pi_{\beta'}\,\delta_1,
  \qquad \pi_{\beta'}:=\sum_{k\ge1}p_k^2,
\]
where $(p_k)_{k\ge1}$ is a Poisson--Dirichlet mass partition with parameters
$(1/\beta',0)$.

It remains to pass from this random limit to the deterministic statement of Theorem \ref{thm:theoremC}. 
Recall that $\mathcal P([0,1])$, the space of Borel probability measures on
$[0,1]$ equipped with the topology of weak convergence, is itself a compact metric (hence
Polish) space, and that, by definition of this topology, a sequence $\mu_n\to\mu$ in
$\mathcal P([0,1])$ if and only if $\mu_n(\varphi)\to\mu(\varphi)$ for every $\varphi\in
C([0,1])$ --- equivalently, the topology of weak convergence is the weakest one making every
evaluation map
\[
\mathrm{ev}_\varphi : \mathcal P([0,1]) \to \mathbb R, \qquad \mathrm{ev}_\varphi(\mu):=\mu(\varphi)=\int_{[0,1]}\varphi\,d\mu,
\]
continuous, for $\varphi\in C([0,1])$ (see e.g.\ Billingsley \cite[Chapter 1]{billingsley99}). Moreover $\mathrm{ev}_\varphi$ is bounded on
$\mathcal P([0,1])$ by $\|\varphi\|_\infty$, since $|\mu(\varphi)|\le\|\varphi\|_\infty\,\mu([0,1])=\|\varphi\|_\infty$,
for every probability measure $\mu$.

Now, $\omega_{n,\beta'}$ and $\omega_{\beta'}:=(1-\pi_{\beta'})\delta_0+\pi_{\beta'}\delta_1$ are
random elements of $\mathcal P([0,1])$, and we have just shown
$\omega_{n,\beta'}\xrightarrow{(\mathrm{law})}\omega_{\beta'}$, as $n\to\infty$. By definition of convergence in law of a random element
of a metric space (here $\mathcal P([0,1])$), this means precisely that $\mathbb
E[F(\omega_{n,\beta'})]\to\mathbb E[F(\omega_{\beta'})]$ for every bounded continuous
$F:\mathcal P([0,1])\to\mathbb R$. Applying this to $F:=\mathrm{ev}_\varphi$, which we have just
seen is bounded and continuous, yields directly
\[
\mathbb E\bigl[\omega_{n,\beta'}(\varphi)\bigr] \ \xrightarrow[n\to\infty]{}\ \mathbb
E\bigl[\omega_{\beta'}(\varphi)\bigr] = \bigl(1-\mathbb E[\pi_{\beta'}]\bigr)\varphi(0) + \mathbb
E[\pi_{\beta'}]\,\varphi(1).
\]
Finally, for a Poisson--Dirichlet partition with parameters $(\alpha,0)$ one has
$\E\left[\sum_kp_k^2\right]=1-\alpha$ (see Pitman and Yor \cite{pitmanyor97}). Here $\alpha=1/\beta'=\beta_c(h)/\beta$, so that
$\E[\pi_{\beta'}]=1-\beta_c(h)/\beta$ and
\[
  \lim_{n\to\infty}\E\left[ \mathcal G^{\otimes2}_{\beta,h,n}\{q_n(\sigma,\sigma')\in\cdot\}\right]
  =\frac{\beta_c(h)}{\beta}\,\delta_0+\left(1-\frac{\beta_c(h)}{\beta}\right)\,\delta_1 ,
\]
which is Theorem \ref{thm:theoremC}. (Remark \ref{rem:PD} is obtained in the same way from \cite[Theorem 4.1]{Mallein2018}, the
Gibbs weights being those of the Poisson--Dirichlet partition with parameter
$1/\beta'=\beta_c(h)/\beta$.) \hfill$\square$


\bibliographystyle{abbrv}
\bibliography{biblio}


\end{document}